\documentclass[11pt]{amsart}
\usepackage{graphicx, amsmath, amssymb, amscd, euscript, amsthm, amsfonts, verbatim, colonequals}
\usepackage[percent]{overpic}
\usepackage{array}
\usepackage{color}

\newcommand*{\defeq}{\mathrel{\vcenter{\baselineskip0.5ex \lineskiplimit0pt\hbox{\scriptsize.}\hbox{\scriptsize.}}}=}

\newtheorem{thm}{Theorem}[section]
\newtheorem{cor}[thm]{Corollary}
\newtheorem{prop}[thm]{Proposition}
\newtheorem{lem}[thm]{Lemma}
\theoremstyle{definition}
\newtheorem{defi}[thm]{Definition}

\theoremstyle{remark}
\newtheorem*{rem}{Remark}
\newtheorem*{ack}{Acknowledgements}

\newtheoremstyle{component}{}{}{}{}{\itshape}{.}{.5em}{\thmnote{#3}#1}
\theoremstyle{component}
\newtheorem*{component}{}

\def\PSL{\operatorname{PSL}}

\renewenvironment{proof}[1][\proofname]{{\noindent \it #1. }}{\qed}

\begin{document}

\title{Figure eight geodesics on 2-orbifolds}
\author{Robert Suzzi Valli}
\address{Department of Mathematics\\
Manhattan College\\
Riverdale, NY}
\email{robert.suzzivalli@manhattan.edu}
\keywords{hyperbolic orbifold, closed geodesic, triangle group}
\subjclass[2010]{57R18 Topology and geometry of orbifolds}

\maketitle
\begin{abstract}
It is known that the shortest non-simple closed geodesic on an orientable hyperbolic 2-orbifold passes through an orbifold point of the orbifold \cite{TN}.  This raises questions about minimal length non-simple closed geodesics disjoint from the orbifold points.  Here we explore once self-intersecting closed geodesics disjoint from the orbifold points of the orbifold, called \textit{figure eight geodesics}.  Using fundamental domains and basic hyperbolic trigonometry we identify and classify all figure eight geodesics on triangle group orbifolds.  This classification allows us to find the shortest figure eight geodesic on a triangle group orbifold, namely the unique one on the (3,3,4)-triangle group orbifold.  We then show that this same curve is the shortest figure eight geodesic on a hyperbolic 2-orbifold without orbifold points of order two.
\end{abstract}

\section{Introduction and results}
The shortest non-simple closed geodesic on a hyperbolic 2-orbifold is the geodesic $\gamma_*$ pictured in Figure \ref{(2,3,7)} which lies on the $(2,3,7)$-triangle group orbifold, denoted by $\mathcal{O}(2,3,7)$ (\cite{TN},\cite{CPNP},\cite{RV}).  The hyperbolic length of $\gamma_*$ is $$\ell_*=2\cosh^{-1}\left[\cos\left(\frac{2\pi}{7}\right)+\frac{1}{2}\right] =0.98399\dots.$$  Note that  $\gamma_*$ runs from the order two orbifold point around the orbifold to the order two orbifold point, and then back along that same path in the reverse direction.  It is natural then to find the shortest non-simple closed geodesics on hyperbolic 2-orbifolds which are disjoint from the orbifold points.  Here we study once self-intersecting closed geodesics disjoint from the orbifold points, which we call \textit{figure eight geodesics}.  By inspecting the fundamental domains of triangle groups we arrive at a classification of all figure eight geodesics on triangle group orbifolds (Theorem \ref{classthm}).  It follows that with the exception of $\mathcal{O}(2,3,r)$ and $\mathcal{O}(2,4,r)$, all triangle group orbifolds contain a figure eight geodesic (Corollary \ref{exceptg}).  With this classification at our disposal, we arrive at the following theorem. 

\begin{thm}
\label{shortfig8tg}
The shortest (unoriented) figure eight geodesic on any triangle group orbifold is the unique one on $\mathcal{O}(3,3,4)$ which bounds the orbifold points of order three and has length $$2\cosh^{-1}\left(\frac{1+\sqrt{2}}{2}\right)=1.26595\dots.$$
\end{thm}

From this result we look to generalize to finding the shortest figure eight geodesic on a hyperbolic 2-orbifold without orbifold points of order two.  Observing that a figure eight geodesic on such a 2-orbifold lies on a type of generalized pair of pants (Proposition \ref{fig8pants}), coupled with the fact that interior lengths decrease as boundary components are shrunk down to punctures  \cite{WT}, we can reduce the problem to looking at triangle group orbifolds.  As this was already established, we have the following. 

\begin{thm}
\label{shortfig8ho}
The shortest (unoriented) figure eight geodesic on a hyperbolic 2-orbifold without orbifold points of order two is the unique one on $\mathcal{O}(3,3,4)$ which bounds the orbifold points of order three and has length $$2\cosh^{-1}\left(\frac{1+\sqrt{2}}{2}\right)=1.26595\dots.$$
\end{thm}

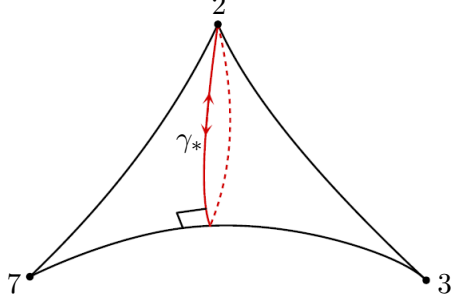
\begin{figure}
\centering
\begin{overpic}[scale=.4]{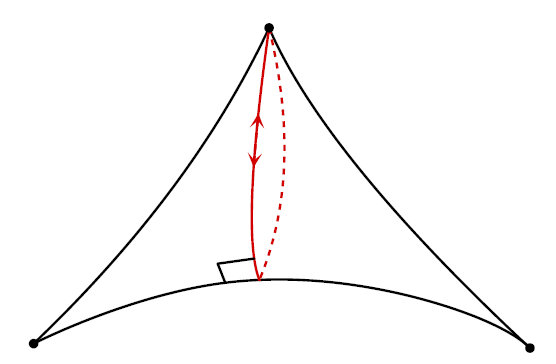}
\put(47,63){2}
\put(98,0){3}
\put(0.7,0){7}
\put(39,33){$\gamma_*$}
\end{overpic}
\caption{Shortest non-simple closed geodesic on a 2-orbifold}
\label{(2,3,7)}
\end{figure}

\indent In recent years, Philippe has written a series of related papers (\cite{EP1}, \cite{EP2}, \cite{EP3}) in which he discusses the length spectrum of triangle groups.  In \cite{EP1} he determined the length of the shortest non-simple closed geodesics on triangle group orbifolds without an orbifold point of order two, and found that the corresponding geodesics are in the free homotopy class of figure eight loops.  It is left open to what the closed geodesics look like, and whether or not they pass through an orbifold point.  These considerations are handled in this note in Theorem \ref{classthm}.\\
\indent  For the basics on hyperbolic geometry, we refer the reader to Beardon \cite{AB1}, Buser \cite{PB}, and Maskit \cite{BM}.  Section 2 gives the necessary development of orbifolds, which is well-known to the experts.  We direct the reader to \cite{JP}, \cite{CHK},\cite{PS} and \cite{TWZ} for further consideration of orbifolds.  Section 3 gives a presentation of triangle groups which is used to prove the classification theorem of figure eight geodesics on triangle group orbifolds (Theorem \ref{classthm}) and Theorem \ref{shortfig8tg}, both of which are contained in Section 4.  Finally, in Section 5 we prove Theorem \ref{shortfig8ho}.

\section{Hyperbolic Orbifolds}
\subsection{Basics}
Let $\mathbb{H}$ denote the hyperbolic plane and suppose $\Gamma$ is a Fuchsian group, that is, a discrete subgroup of the group of orientation-preserving isometries of the hyperbolic plane, which we can identify with the matrix group PSL$(2,\mathbb{R})=$ SL$(2,\mathbb{R})/\{\pm I\}$.  The group $\Gamma$ induces the continuous, open projection $$\pi:\mathbb{H}\rightarrow\mathbb{H}/\Gamma,$$ where $\pi(z)=\Gamma(z)$, the orbit of $z\in\mathbb{H}$ under $\Gamma$.  We are interested in the case where $\Gamma$ contains finite-order elliptic elements.  In this case the map $\pi$ is a branched covering map and the quotient $\mathbb{H}/\Gamma$ is a 2-dimensional orientable hyperbolic orbifold.  From now on, for simplicity, we will write 2-orbifold, keeping in mind that the orbifolds we consider here will always be orientable and hyperbolic.\\
\indent If $g\in\Gamma$ is an elliptic element with fixed point $\tilde{x}\in\mathbb{H}$, then the projection $\pi(\tilde{x})$ is an \textit{orbifold point} on $\mathbb{H}/\Gamma$ whose order is the cardinality of $\Gamma_{\tilde{x}}=\{f\in\Gamma\mid f(\tilde{x})=\tilde{x}\}$.  If $\Sigma$ represents the set of orbifold points on the 2-orbifold $\mathbb{H}/\Gamma$, then the branched covering map $\pi$ possesses the following properties:
\begin{itemize}
\item[(i)]
$\pi :\mathbb{H} - \pi^{-1}(\Sigma)\rightarrow X - \Sigma$ is a covering map, and
\item[(ii)]
for all $\tilde{x}\in\pi^{-1}(\Sigma)$, there exists a hyperbolic disc $D$ centered at $\tilde{x}$ such that $D$ is \textit{precisely invariant} under $\Gamma_{\tilde{x}}$ in $\Gamma$ (that is, for all $g\in\Gamma_{\tilde{x}}$, $g(D)=D$ and for all $g\in\Gamma - \Gamma_{\tilde{x}}$, $g(D)\cap D =\varnothing$).
\end{itemize}

\indent Let $\Gamma$ be a finitely generated Fuchsian group. It follows that $\Gamma$ will have finitely many conjugacy classes of maximal parabolic, elliptic, and boundary hyperbolic cyclic subgroups.  Suppose $\mathbb{H}/\Gamma$ has genus $g$, and $\Gamma$ contains $c$ conjugacy classes of maximal elliptic cyclic subgroups of orders $m_1,\dots,m_c\geq2$, $n$ conjugacy classes of maximal parabolic cyclic subgroups, and $b$ conjugacy classes of maximal boundary hyperbolic cyclic subgroups.  Then, we say $\Gamma$ has \textit{signature} $(g:m_1,\dots,m_c;n;b)$.  The import is that if we suppose $\Gamma$ has this signature, then the resulting 2-orbifold $\mathbb{H}/\Gamma$ has genus $g$, $c$ orbifold points of orders $m_1,\dots,m_c$, $n$ punctures, and $b$ boundary components.\\

\subsection{Paths and Homotopy}
\label{pathshomotopy}
The goal of this section is to establish an association between homotopy classes of loops on $\mathbb{H}/\Gamma$ and elements of the Fuchsian group $\Gamma$. The first step is to define the notion of a path on $\mathbb{H}/\Gamma$.  It will be crucial that we have the uniqueness of path liftings property from covering space theory.\\  
\indent A point on $\mathbb{H}/\Gamma$ which is not an orbifold point will be called a \textit{regular point}.  Fix a regular point $p\in\mathbb{H}/\Gamma$ and some lift $\widetilde{p}\in\mathbb{H}$ of $p$ to be base points.  Note that we cannot simply deem a continuous map $\alpha:[0,1]\rightarrow\mathbb{H}/\Gamma$, $\alpha(0)=p$, to be a path on $\mathbb{H}/\Gamma$ since if the path passes through an orbifold point we would not have a unique lift of $\alpha$ to $\mathbb{H}$ starting at $\widetilde{p}$.  For example, in Figure \ref{orbpath} we have the group $\Gamma=\langle z\mapsto iz\rangle$ acting on the unit disk model of $\mathbb{H}$.  The path $\alpha$ passes through the orbifold point of order four and we see the four possible continuous lifts to $\mathbb{H}$.  In general, if $\alpha$ passes through a single orbifold point of order $n$, then there are $n$ choices of a continuous lift of $\alpha$ starting at $\widetilde{p}$.

\begin{figure}
\centering
\begin{overpic}[scale=.5]{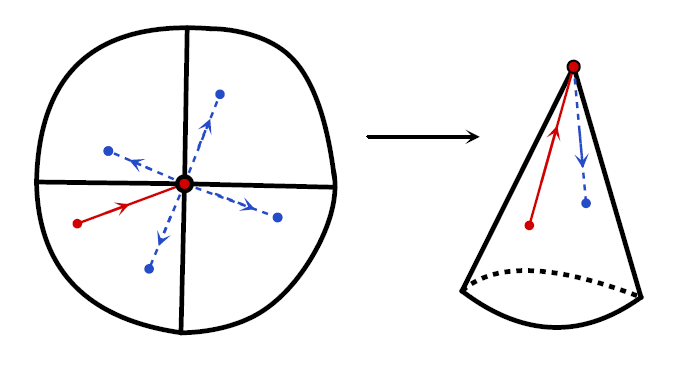}
\put(5,49){$\mathbb{H}$}
\put(11,17){$\widetilde{p}$}
\put(60,36){$\pi$}
\put(94,49){$\mathbb{H}/\Gamma$}
\put(76,17){$p$}
\put(80,25){$\alpha$}
\end{overpic}
\caption{A path passing through an orbifold point of order four and its lifts}
\label{orbpath}
\end{figure}
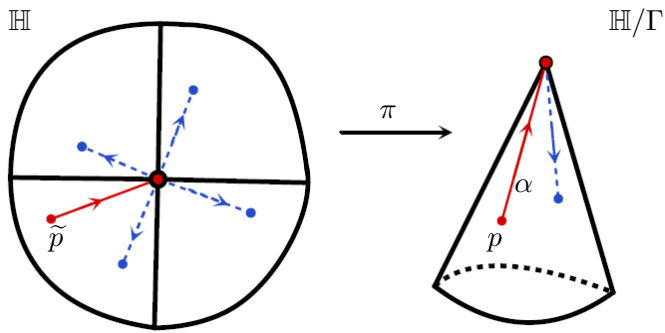

In order to guarantee uniqueness of path liftings we make the following definition.

\begin{defi}
A \textit{path} in $\mathbb{H}/\Gamma$ consists of a pair $(\alpha,\widetilde{\alpha})$, in which
\begin{itemize}
 \item[(1)]
 $\alpha:[0,1]\rightarrow\mathbb{H}/\Gamma$, is a continuous map with $\alpha(0)=p$, and there are at most finitely many $t\in[0,1]$ such that $\alpha(t)$ is an orbifold point; and
 \item[(2)]
 $\widetilde{\alpha}:[0,1]\rightarrow\mathbb{H}$ is some continuous lift of $\alpha$ to $\mathbb{H}$ with $\widetilde{\alpha}(0)=\widetilde{p}$.
 \end{itemize}
 If the context is clear we will denote the pair as simply $\alpha$.
\end{defi}

\begin{rem}
If $\alpha$ does not pass through an orbifold point then $\widetilde{\alpha}$ is uniquely defined since the map $\pi$ is a covering map in the complement of the elliptic fixed points, and thus has the uniqueness of path liftings property.
\end{rem}

\begin{defi}
Let $\alpha$ and $\beta$ be two paths in $\mathbb{H}/\Gamma$.  We say that $\alpha$ and $\beta$ are \textit{$\mathbb{H}/\Gamma$-homotopic}, denoted $\alpha\sim\beta$, if their respective lifts $\widetilde{\alpha}$ and $\widetilde{\beta}$ have the same end point.
\end{defi}

Ultimately $\mathbb{H}/\Gamma$-homotopy is the ``usual" notion of homotopy in the complement of the orbifold points, but we need to take a closer look at what happens to a path as it is $\mathbb{H}/\Gamma$-homotoped passed an orbifold point.  We consider an example for an elliptic isometry of order three in Figure \ref{elliptic3}, but it generalizes to an elliptic isometry of any finite order (see \cite{CG}).  As the definition indicates homotopic paths in $\mathbb{H}$ project to $\mathbb{H}/\Gamma$-homotopic paths in $\mathbb{H}/\Gamma$.  Thus, the top row of Figure \ref{elliptic3} represents three homotopic paths in $\mathbb{H}$ with end points on the boundary of a hyperbolic disc centered at the origin, which we normalize to be the fixed point of an elliptic isometry of order three.  Below each disc is an aerial view of the corresponding $\mathbb{H}/\Gamma$-homotopic paths on the resulting cone.  Hence, we see that a path going around the orbifold point of order three once in one direction is $\mathbb{H}/\Gamma$-homotopic to a path going around the orbifold point twice in the opposite direction.  This observation is closely related to the order of the orbifold point and will become more transparent as we establish the relationship between homotopy classes of loops and elements of the Fuchsian group.\\
\indent Our primary concern is closed paths.  A \textit{loop} in $\mathbb{H}/\Gamma$ based at $p$ is a path $\alpha$ in $\mathbb{H}/\Gamma$ such that $\alpha(0)=\alpha(1)=p$.

\begin{figure}
\centering
\begin{overpic}[scale=.4]{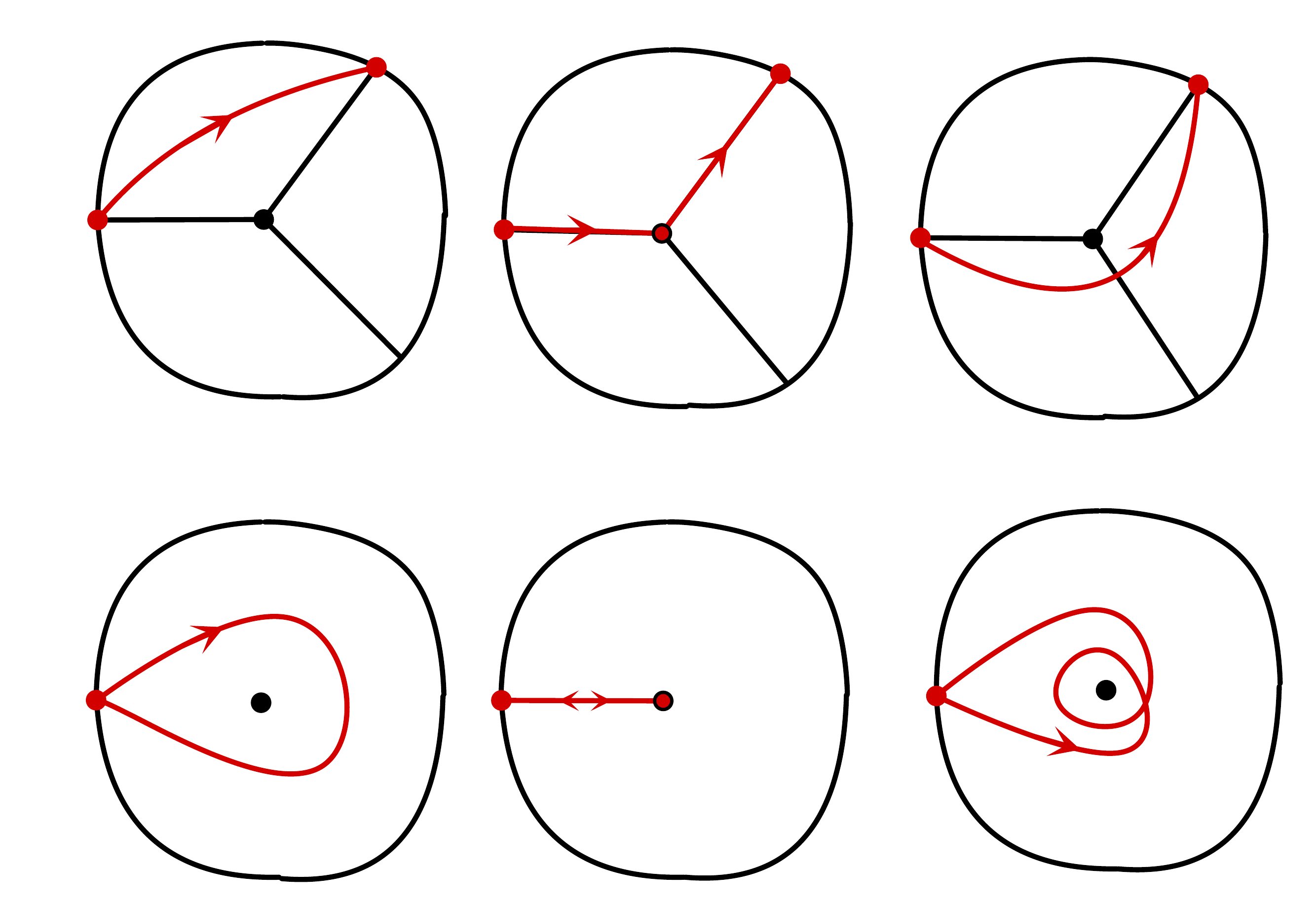}
\put(2,65){$\mathbb{H}$}
\put(2,29){$\mathbb{H}/\Gamma$}
\put(19,33){$\downarrow$}
\put(49,33){$\downarrow$}
\put(83,33){$\downarrow$}
\put(16,33){$\pi$}
\put(4,52){$\widetilde{p}$}
\put(4,15.5){$p$}
\put(34.2,50){$\sim$}
\put(65.5,50){$\sim$}
\put(34.2,15){$\sim$}
\put(66,15){$\sim$}
\end{overpic}
\caption{$\mathbb{H}/\Gamma$-homotopic paths about an orbifold point of order three}
\label{elliptic3}
\end{figure}
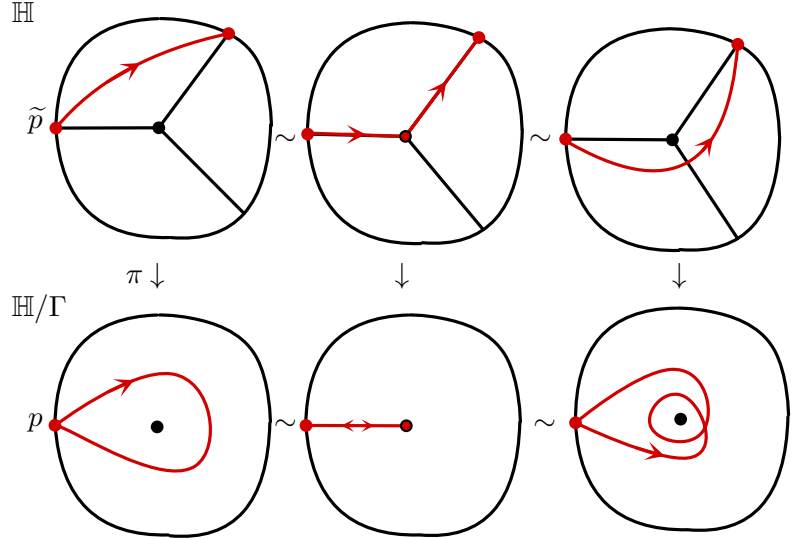

\begin{rem}
If $\alpha$ is a loop based at $p$, then the end point of its lift $\widetilde{\alpha}$ is $g(\widetilde{p})$, for some $g\in\Gamma$.  In fact, $g$ is unique since we chose $p$ to be a regular point.
\end{rem}

$\mathbb{H}/\Gamma$-homotopy induces an equivalence relation on the set of loops in $\mathbb{H}/\Gamma$ based at $p$.  Suppose $(\alpha,\widetilde{\alpha})$ and $(\beta,\widetilde{\beta})$ are two loops in $\mathbb{H}/\Gamma$ based at $p$ and furthermore, suppose the end point of $\widetilde{\alpha}$ is $g(\widetilde{p})$, for some $g\in\Gamma$.  The \textit{concatenation} of $(\alpha,\widetilde{\alpha})$ with $(\beta,\widetilde{\beta})$, denoted $(\alpha,\widetilde{\alpha})*(\beta,\widetilde{\beta})$ is defined to be the loop $(\alpha*\beta,\widetilde{\alpha}*g\widetilde{\beta})$, where

\begin{equation*}
\alpha*\beta=\left\{\begin{array}{rl}
\alpha(2t) & ,  0\leq t\leq 1/2\\
\beta(2t-1) & , 1/2\leq t\leq1\
\end{array}\right.,
\quad
\widetilde{\alpha}*g\widetilde{\beta}=\left\{\begin{array}{rl}
\widetilde{\alpha}(2t) & ,  0\leq t\leq 1/2\\
g\widetilde{\beta}(2t-1) & , 1/2\leq t\leq1\
\end{array}\right..
\end{equation*}

We remark that we have defined concatenation of loops to be performed from left to right, while function composition will be performed from right to left.  Concatenation of loops respects $\mathbb{H}/\Gamma$-homotopy, that is, if $\alpha\sim\alpha^{\prime}$ and $\beta\sim\beta^{\prime}$, then $\alpha*\beta\sim\alpha^{\prime}*\beta^{\prime}$.  Hence we define the \textit{orbifold fundamental group} of $\mathbb{H}/\Gamma$ based at $p$, denoted $\Pi_1(\mathbb{H}/\Gamma,p)$, to be the set of $\mathbb{H}/\Gamma$-homotopy classes of loops based at $p$ under the operation of concatenation.  It can be shown that $\Pi_1(\mathbb{H}/\Gamma,p)$ is indeed a group.  We now come to the main purpose for the above definitions, which has an analagous proof to the 2-manifold case.

\begin{thm}
\label{isomorphthm}
Suppose $\Gamma$ is a Fuchsian group and $\pi:\mathbb{H}\rightarrow\mathbb{H}/\Gamma$ is the projection map.  Fix a regular point $p\in\mathbb{H}/\Gamma$ and a lift $\widetilde{p}\in\pi^{-1}(p)$.  Then $\Pi_1(\mathbb{H}/\Gamma,p)$ is isomorphic to $\Gamma$.
\end{thm}

\subsection{Closed Geodesics}
Let $g$ be a hyperbolic isometry in the Fuchsian group $\Gamma$ with axis $\mathcal{A}_g$.  If $\pi:\mathbb{H}\rightarrow\mathbb{H}/\Gamma$ is the projection map, then $\pi(\mathcal{A}_g)$ is a \textit{closed geodesic} on $\mathbb{H}/\Gamma$.  We remark that in the case that $\pi(\mathcal{A}_g)$ passes through an orbifold point, the resulting geodesic is a piecewise closed geodesic with smoothness lacking at the orbifold points.  For our purposes we will still refer to such phenomena as closed geodesics.  Furthermore, unless otherwise stated we will assume all closed geodesics are primitive.\\
\indent Fix regular points $p$ and $q$, as well as lifts $\widetilde{p}$ and $\widetilde{q}$ of $p$ and $q$, respectively.  Suppose $\alpha$ and $\beta$ are loops on $\mathbb{H}/\Gamma$ where $\alpha$ is based at $p$ and $\beta$ is based at $q$.  We say that $\alpha$ is \textit{freely $\mathbb{H}/\Gamma$-homotopic} to $\beta$ if and only if there exists a path $\delta$ on $\mathbb{H}/\Gamma$ which starts at $p$, ends at $q$, and is disjoint from the orbifold points of $\mathbb{H}/\Gamma$, so that $\alpha$ is $\mathbb{H}/\Gamma$-homotopic to $\delta*\beta*\delta^{-1}$, that is, the end points of the lifts of $\alpha$ and $\delta*\beta*\delta^{-1}$ starting at $\widetilde{p}$ coincide.

\indent A loop $\alpha$ on $\mathbb{H}/\Gamma$ is called \textit{essential} if it does not bound a topological disc on $\mathbb{H}/\Gamma$.

\begin{thm}
\label{uniquegeodthm}
Suppose $\Gamma$ is a Fuchsian group and $\pi:\mathbb{H}\rightarrow\mathbb{H}/\Gamma$ is the projection map.  Fix a regular point $p\in\mathbb{H}/\Gamma$ and a lift $\widetilde{p}\in\pi^{-1}(p)$.  If $\alpha$ is an essential loop on $\mathbb{H}/\Gamma$ based at $p$, then exactly one of the following holds.
\begin{itemize}
\item[(1)]
There exists a unique closed geodesic in the free $\mathbb{H}/\Gamma$-homotopy class of $\alpha$.
\item[(2)]
There is no closed geodesic in the free $\mathbb{H}/\Gamma$-homotopy class of $\alpha$; furthermore, there is a simple loop in the free $\mathbb{H}/\Gamma$-homotopy class of $\alpha$ that bounds a puncture.
\item[(3)]
There is no closed geodesic in the free $\mathbb{H}/\Gamma$-homotopy class of $\alpha$; furthermore, there is a simple loop in the free $\mathbb{H}/\Gamma$-homotopy class of $\alpha$ that bounds an orbifold point.

\end{itemize}
\end{thm}

\begin{proof}
Let $\Phi:\Pi_1(\mathbb{H}/\Gamma,p)\rightarrow\Gamma$ be the isomorphism guaranteed by Theorem \ref{isomorphthm} and suppose $\Phi([\alpha])=g$.  The three cases arise respectively from whether $g$ is hyperbolic, parabolic, or elliptic.  As the proof is very similar to the 2-manifold case, we only verify the existence part of (1).  So suppose $g$ is hyperbolic.  Choose a path $\widetilde{\delta}$ in $\mathbb{H}$ which starts at $\widetilde{p}$, ends at some point $\widetilde{q}\in\mathcal{A}_g$, and does not pass through any elliptic fixed points.  Let $\widetilde{\gamma}$ be the path on $\mathcal{A}_g$ from $\widetilde{q}$ to $g(\widetilde{q})$.  It follows that $g(\widetilde{\delta}^{-1})$ is a path from $g(\widetilde{q})$ to $g(\widetilde{p})$.  Thus, the paths $\widetilde{\alpha}$ and $\widetilde{\delta}*\widetilde{\gamma}*g(\widetilde{\delta}^{-1})$ are homotopic paths in $\mathbb{H}$.  Letting $\delta=\pi(\widetilde{\delta})$ and $\gamma=\pi(\widetilde{\gamma})=\pi(\mathcal{A}_g)$, we have that $\alpha$ is $\mathbb{H}/\Gamma$-homotopic to $\delta*\gamma*\delta^{-1}$.  Thus, $\alpha$ is freely $\mathbb{H}/\Gamma$-homotopic to the closed geodesic $\gamma$.  
\end{proof}

\begin{defi}
Let $\Gamma$ be a Fuchsian group and suppose $\gamma$ is a closed geodesic on $\mathbb{H}/\Gamma$.  Lifting one copy of $\gamma$ yields a geodesic segment $\widetilde{\gamma}$ lying on the axis $\mathcal{A}_g$, of some hyperbolic isometry $g\in\Gamma$.  Now consider all axes $\{\mathcal{A}_{h_i}\}$ of hyperbolic isometries $h_i$ conjugate to $g$ in $\Gamma$ which intersect $\widetilde{\gamma}$ transversely.  We define the \textit{self-intersection number} of $\gamma$ to be the positive integer $|\{\mathcal{A}_{h_i}\}|$.  Suppose $\gamma$ has self-intersection number $n$.  If $n=0$, then we call $\gamma$ a \textit{simple closed geodesic}, while if $n>0$, then we call $\gamma$ a \textit{non-simple closed geodesic}.
\end{defi}

\section{Triangle Groups}
\label{tgsection}
Suppose $p,q,$ and $r$ are positive integers satisfying $\frac{1}{p}+\frac{1}{q}+\frac{1}{r}<1$. Then let $T$ be a hyperbolic triangle in $\mathbb{H}$ with interior angles $\frac{\pi}{p}, \frac{\pi}{q}$,  and $\frac{\pi}{r}$. Consider the discrete group $G$ generated by the reflections in the geodesics containing the sides of $T$.  The subgroup of orientation-preserving isometries of $G$, denoted $\Gamma(p,q,r)$, is referred to as a \textit{$(p,q,r)$-triangle group}.  We call a Fuchsian group $\Gamma$ a \textit{triangle group} if it is a $(p,q,r)$-triangle group for some integers $p,q$ and $r$.\\
\indent  In order to obtain a group presentation of $\Gamma(p,q,r)$ which will make computations easier, we use the same normalizations and notation as in \cite{TN}.  We work in the upper-half plane model of $\mathbb{H}$ and normalize $T$ to be the triangle in the left half of Figure \ref{fd(p,q,r)}.  Two vertices of $T$ are $i$ and  $\lambda^{-1}i$, for some determined $\lambda=\lambda(p,q,r)>1$, and $T$ has sides labeled $s_1,s_2,s_3$. Let $\sigma_i$ denote the reflection in the geodesic containing side $s_i$, for $i=1,2,3$.  Thus $G=\langle\sigma_1,\sigma_2,\sigma_3\rangle$ and the subgroup $\Gamma(p,q,r)$ consists of all the even length words in the $\sigma_i$.  It follows that $\Gamma(p,q,r)$ is index two in $G$ and so $T^*=T\cup\sigma_1(T)$ is a fundamental domain for $\Gamma(p,q,r)$.\\

\begin{figure}
\centering
\begin{overpic}[scale=.4]{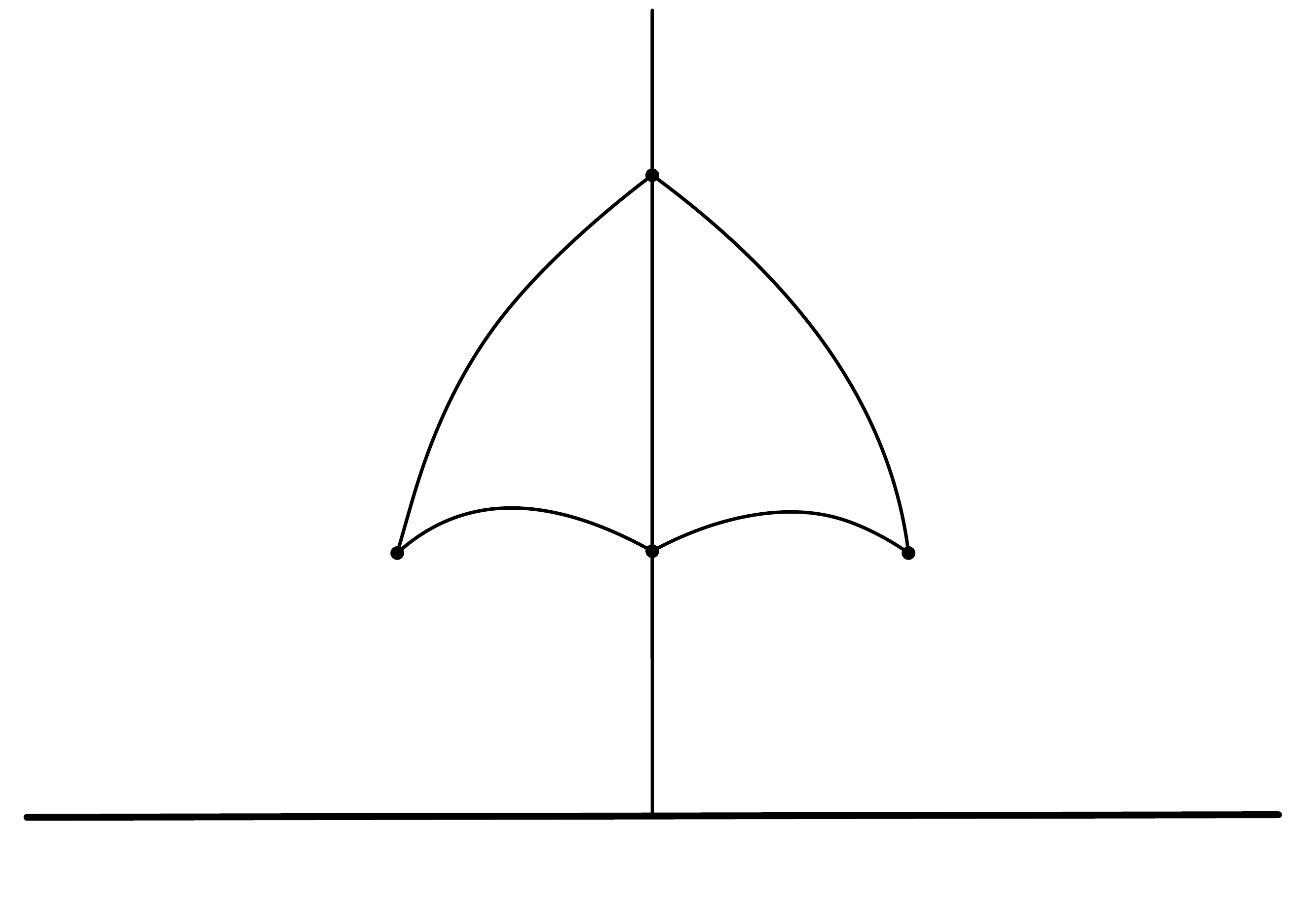}
\put(46,50){$\frac{\pi}{p}$}
\put(50,50){$\frac{\pi}{p}$}
\put(31.5,30){$\frac{\pi}{r}$}
\put(46.5,30){$\frac{\pi}{q}$}
\put(50.5,30){$\frac{\pi}{q}$}
\put(65.8,30){$\frac{\pi}{r}$}
\put(34,45){$s_2$}
\put(39,28){$s_3$}
\put(50,40){$s_1$}
\put(51,55.5){$i$}
\put(51,24){\footnotesize$\lambda^{-1}i$}
\put(42,40){\footnotesize$T$}
\put(55,40){\footnotesize$\sigma_1(T)$}
\put(49,3.5){0}
\put(80,60){$\mathbb{H}$}
\end{overpic}
\caption{Fundamental domain $T^*=T\cup\sigma_1(T)$ for $\Gamma(p,q,r)$}
\label{fd(p,q,r)}
\end{figure}
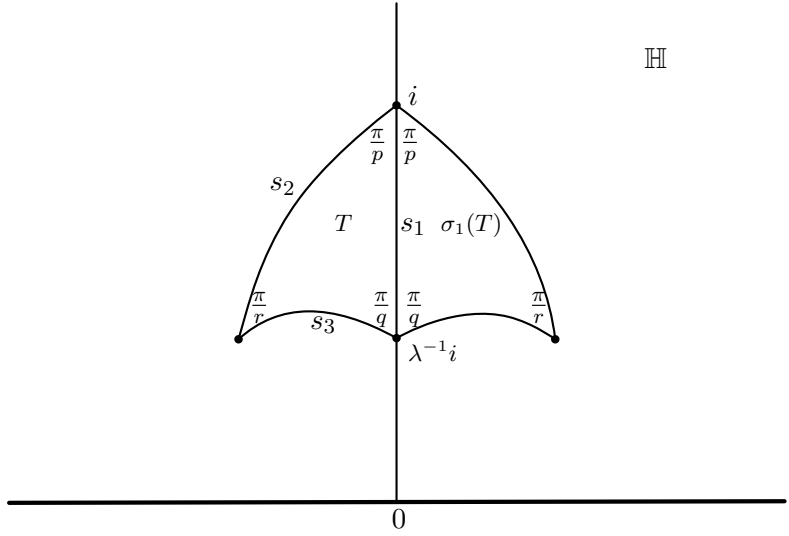
 \indent The group $\Gamma(p,q,r)$ can be generated by the two elements $A=\sigma_2\sigma_1$ and $B=\sigma_1\sigma_3$.   Both $A$ and $B$ can be expressed as elements of SL$(2,\mathbb{R})$:\\

\begin{center}
$A=\left( {\begin{array}{cc}
 \cos(\frac{\pi}{p}) & -\sin(\frac{\pi}{p}) \\
 \sin(\frac{\pi}{p}) & \cos(\frac{\pi}{p})  \\
 \end{array} } \right)$, \quad
 $B=\left( {\begin{array}{cc}
 \cos(\frac{\pi}{q}) & -\lambda^{-1}\sin(\frac{\pi}{q}) \\
 \lambda\sin(\frac{\pi}{q}) & \cos(\frac{\pi}{q})  \\
 \end{array} } \right)$\vspace*{.1in}
\end{center}

Let $C=B^{-1}A^{-1}=\sigma_3\sigma_1\sigma_1\sigma_2=\sigma_3\sigma_2$.  Then $C$ is an elliptic isometry of order $r$ whose fixed point is the vertex of $T^*$ to the left of the imaginary axis.  Using the fundamental domain in Figure \ref{fd(p,q,r)} we can write down the following presentation for $\Gamma(p,q,r)$:

\begin{equation}
\label{prestg}
\Gamma(p,q,r)=\langle A,B\mid A^p=B^q=(B^{-1}A^{-1})^r=I\rangle
\end{equation}

It follows that the quotient $\mathbb{H}/\Gamma(p,q,r)\colonequals\mathcal{O}(p,q,r)$ is a sphere with three orbifold points of orders $p$, $q$ and $r$.  We refer to $\mathcal{O}(p,q,r)$ as a \textit{$(p,q,r)$-triangle group orbifold}. If $\Gamma$ is a triangle group, we call $\mathbb{H}/\Gamma$ a \textit{triangle group orbifold}.\\
\indent Our next goal is to determine $\lambda$.  Using either triangle in Figure \ref{fd(p,q,r)}, we can arrive at the following equation using the hyperbolic law of cosines:

 \begin{equation}
 \label{sumlambda}
 \lambda+\lambda^{-1}=\dfrac{2E}{\sin\left(\frac{\pi}{p}\right)\sin\left(\frac{\pi}{q}\right)},
 \end{equation}

 \noindent where
 \begin{equation}
 \label{E}
 E=\cos\left(\frac{\pi}{r}\right)+\cos\left(\frac{\pi}{p}\right)\cos\left(\frac{\pi}{q}\right).
 \end{equation}

 \noindent Solving (\ref{sumlambda}) for $\lambda$ we obtain

 \begin{equation*}
 \lambda=\lambda(p,q,r)=\dfrac{E+\left(E^2-\sin^2\left(\frac{\pi}{p}\right)\sin^2\left(\frac{\pi}{q}\right)\right)^{\frac{1}{2}}}{\sin\left(\frac{\pi}{p}\right)\sin\left(\frac{\pi}{q}\right)}.
 \end{equation*}

\noindent  Another fact we will need later is that the trace of the matrix $C$, which we denote by $tr(C)$, is negative.  By multiplying $B^{-1}\cdot A^{-1}$ and using (\ref{sumlambda}) we have:

 \begin{equation}
 \label{trC}
 tr(C)=2\cos\left(\frac{\pi}{p}\right)\cos\left(\frac{\pi}{q}\right)-(\lambda+\lambda^{-1})\sin\left(\frac{\pi}{p}\right)\sin\left(\frac{\pi}{q}\right)=-2\cos\left(\frac{\pi}{r}\right)
\end{equation}

 In our discussion we allow triangle groups to contain parabolic isometries.  To achieve this we use the presentation (\ref{prestg}) and consider the elliptic isometry $C$ while letting $r\rightarrow\infty$.  In the limit $C$ is a parabolic isometry, and if we apply this limit ($r\rightarrow\infty$) considering the entire group $\Gamma(p,q,r)$ from (\ref{prestg}), we get the group $\Gamma(p,q,\infty)=\langle A,B\mid A^p=B^q=I\rangle$, where $\infty$ represents the conjugacy class of the maximal parabolic cyclic subgroup $\langle C\rangle$.  It follows that $\mathcal{O}(p,q,\infty)$ is a sphere with a puncture and two orbifold points having orders $p$ and $q$.  Similarly, if we start with the presentation (\ref{prestg}) and let $q,r\rightarrow\infty$ our group now becomes $\Gamma(p,\infty,\infty)=\langle A,B\mid A^p=I\rangle$.  Finally, we consider $\Gamma(\infty,\infty,\infty)$.  We treat this case independently from the others since $A\rightarrow I$ as $p\rightarrow\infty$, and $I$ is not a parabolic isometry.  So, let $\Gamma(\infty,\infty,\infty)=\langle A,B\rangle$ where

\begin{equation*}
\label{prestgp3}
A=\left( {\begin{array}{cc}
 1 & 2 \\
 0 & 1\\
 \end{array} } \right), \quad
 B=\left( {\begin{array}{cc}
 1 & 0 \\
 -2 & 1  \\
 \end{array}}\right).\end{equation*}\\

\begin{rem} $\Gamma(p,q,r)$ is conjugate to $\Gamma(p^{\prime},q^{\prime},r^{\prime})$ in the full isometry group of $\mathbb{H}$ if and only if $(p^{\prime},q^{\prime},r^{\prime})$ is a permutation of $(p,q,r)$.  Thus, without loss of generality we will assume throughout that $2\leq p\leq q\leq r\leq\infty$.\\
\end{rem}

\section{Figure Eight Geodesics on Triangle Group Orbifolds}
\subsection{Classifying Figure Eight Geodesics}
\indent Let $\Gamma(p,q,r)$ be the normalized triangle group as in Section \ref{tgsection}, and furthermore, assume that $2\leq p\leq q\leq r\leq\infty$.  Let $\pi:\mathbb{H}\rightarrow \mathcal{O}(p,q,r)=\mathbb{H}/\Gamma(p,q,r)$ be the projection map.  We remind the reader that concatenation of paths is performed from left to right, while function composition is from right to left.\\
\indent First, decompose $\mathcal{O}(p,q,r)$ into two isometric triangles by drawing in the geodesic paths between each orbifold point (the black segments in Figure \ref{basisfundgp}).  Orient the ``front" triangle clockwise and similarly orient the two triangles in the lift of $\mathcal{O}(p,q,r)$ pictured in Figure \ref{fd(p,q,r)} to be clockwise.  This establishes an orientation on $\mathcal{O}(p,q,r)$.  Let $x$ in Figure \ref{basisfundgp} be the base point of $\mathcal{O}(p,q,r)$ and choose the base point $\widetilde{x}\in\mathbb{H}$ to be the lift of $x$ which is on the imaginary axis between $i$ and $\lambda^{-1}i$.  In Figure \ref{basisfundgp} consider the three oriented loops $\alpha,\beta$, and $\gamma$ on $\mathcal{O}(p,q,r)$ based at $x$.  Letting $\Phi:\Pi_1(\mathcal{O}(p,q,r),x)\rightarrow\Gamma(p,q,r)$ be the corresponding isomorphism, it follows that $\Phi([\alpha])=A$, $\Phi([\beta])=B$, and $\Phi([\gamma])=B^{-1}A^{-1}=C$, where $A,B$ and $C$ are as in Section \ref{tgsection}.  Now consider the following three loops on $\mathcal{O}(p,q,r)$: $\mathcal{C}_{pq}\colonequals\beta*\alpha^{-1}$, $\mathcal{C}_{qr}\colonequals\beta*\gamma^{-1}$, and $\mathcal{C}_{pr}\colonequals\alpha*\gamma^{-1}$.  Hence, $\Phi([\mathcal{C}_{pq}])=BA^{-1}$, $\Phi([\mathcal{C}_{qr}])=BC^{-1}$, and $\Phi([\mathcal{C}_{pr}])=AC^{-1}$.  It follows from Theorem \ref{uniquegeodthm} that the loops $\mathcal{C}_{pq}, \mathcal{C}_{qr},$ and $\mathcal{C}_{pr}$ are respectively, freely $\mathcal{O}(p,q,r)$-homotopic to either a loop bounding the projection under $\pi$ of the fixed point (if elliptic or parabolic) or the projection of the axis (if hyperbolic) of $BA^{-1},BC^{-1},$ and $AC^{-1}$.  By abuse of notation, we use $\mathcal{C}_{pq}, \mathcal{C}_{qr},$ and $\mathcal{C}_{pr}$ to mean the free $\mathcal{O}(p,q,r)$-homotopy class of $\beta*\alpha^{-1}, \beta*\gamma^{-1}$, and $\alpha*\gamma^{-1}$, respectively. \\
\indent Theorem \ref{classthm} of this section explicity describes what these projections look like in the hyperbolic case, with the goal of identifying all figure eight geodesics on triangle group orbifolds.  The key fact to observe is that a figure eight geodesic on a triangle group orbifold must be in at least one of the free $\mathcal{O}(p,q,r)$-homotopy classes $\mathcal{C}_{pq}, \mathcal{C}_{qr}$, $\mathcal{C}_{pr}$, or the classes corresponding to the opposite orientations.  Before stating and proving the theorem we need to establish the following two lemmas.  

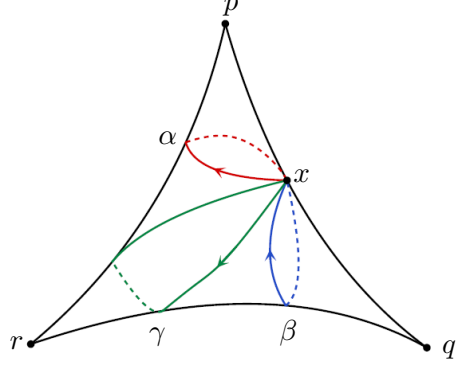
\begin{figure}
\centering
\begin{overpic}[scale=.4]{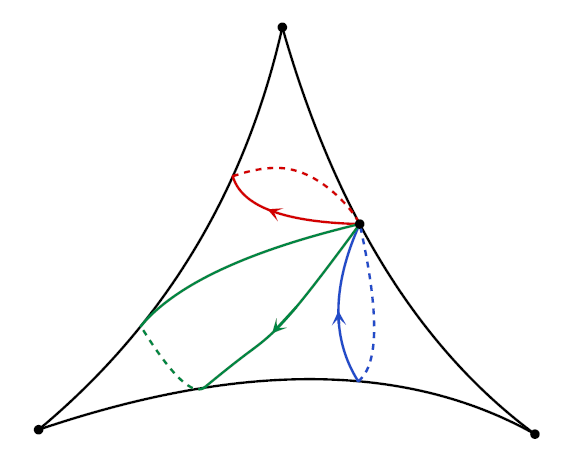}
\put(2,6){$r$}
\put(95,5){$q$}
\put(48,79){$p$}
\put(63,42){$x$}
\put(34,50){$\alpha$}
\put(60,8){$\beta$}
\put(32,8){$\gamma$}
\end{overpic}
\caption{Loops on $\mathcal{O}(p,q,r)$ based at $x$}
\label{basisfundgp}
\end{figure}

\begin{lem}
\label{imaglem}
 Suppose $g$ is a hyperbolic isometry represented by $\left( {\begin{smallmatrix}
 a & b \\
 c & d  \\
 \end{smallmatrix} } \right)\in\PSL(2,\mathbb{R})$ where $bc>0$.  Then its axis $\mathcal{A}_g$ intersects the imaginary axis at the point $i\sqrt{\frac{b}{c}}$.
 \end{lem}

 \begin{proof}
 The fixed points of $g$ are $\frac{a-d\pm\sqrt{(a+d)^2-4}}{2c}$.  So the geodesic $\mathcal{A}_g$ has corresponding equation $\left|z-\frac{a-d}{2c}\right|=\frac{\sqrt{(a+d)^2-4}}{2c}$, where $z\in\mathbb{U}$. Letting $z=iy$ in this equation it follows from a straightforward computation that $ y=\sqrt{\frac{b}{c}}$.
Since $bc>0$, $b$ and $c$ have the same sign and thus $\sqrt{\frac{b}{c}}$ is a real number.  So $\mathcal{A}_g$ intersects the imaginary axis at the point $i\sqrt{\frac{b}{c}}$.
 \end{proof}

\begin{lem}
\label{BAlem}
The element $BA^{-1}\in\Gamma(3,q,r)$ for $q\geq3,r\geq4$, is hyperbolic and its axis intersects the imaginary axis at the point $ti$, where $t\leq1$. Moreover, $t=1$ if and only if $q=r$.
 \end{lem}

 \begin{proof}
 To show that $BA^{-1}$ is hyperbolic we will show that its trace is greater than two.  We follow the development in \cite{IR} which yields a useful trace identity. By the Cayley-Hamilton Theorem we have the following identity which holds for determinant one matrices: $ A^{-1}=tr(A)\cdot I - A$.  It follows that $A^{-1}B=tr(A)\cdot B-AB$ and so,
 \begin{equation}
 \label{treqn}
 tr(A^{-1}B)=tr(A)tr(B)-tr(AB).
 \end{equation}
 Considering $A,B,C\in\Gamma(p,q,r)$ where $\Gamma(p,q,r)$ has presentation (\ref{prestg}), we see that $tr(A)=2\cos(\frac{\pi}{p})$, $tr(B)=2\cos(\frac{\pi}{q})$ and we have $tr(AB)=tr(C^{-1})=tr(C)=-2\cos(\frac{\pi}{r})$ from (\ref{trC}).  Thus, with (\ref{treqn}) and since we are considering the case where $p=3$, $q\geq3$, and $r\geq4$ we get \begin{eqnarray*}
tr(BA^{-1}) = tr(A^{-1}B)&=&tr(A)tr(B)-tr(AB)\\
 &\geq& 2\left[\frac{1}{2}+\frac{\sqrt{2}}{2}\right]=1+\sqrt{2}>2
 \end{eqnarray*}

Thus, $BA^{-1}\in\Gamma(3,q,r)$ (for $q\geq3, r\geq4$) is hyperbolic.  Using (\ref{prestg}), we have
 \begin{equation*}
 \label{matrixBA}BA^{-1}=\left(\begin{matrix}
 \frac{1}{2}\cos(\frac{\pi}{q})+\frac{\sqrt{3}}{2}\lambda^{-1}\sin(\frac{\pi}{q}) & \frac{\sqrt{3}}{2}\cos(\frac{\pi}{q})-\frac{1}{2}\lambda^{-1}\sin(\frac{\pi}{q})\\
 \frac{1}{2}\lambda\sin(\frac{\pi}{q})-\frac{\sqrt{3}}{2}\cos(\frac{\pi}{q}) & \frac{\sqrt{3}}{2}\lambda\sin(\frac{\pi}{q})+\frac{1}{2}\cos(\frac{\pi}{q})\\
 \end{matrix}\right)\defeq\left(\begin{matrix}
 a & b\\
 c & d\\
 \end{matrix}\right)
 \end{equation*}

 In light of Lemma \ref{imaglem}, it suffices to show that $b,c>0$ and $\sqrt{\frac{b}{c}}\leq1$, with equality holding if and only if $q=r$.  It follows that

 \begin{eqnarray*}
b>0 
&\Leftrightarrow& \sqrt{3}\lambda\cos\left(\frac{\pi}{q}\right)>\sin\left(\frac{\pi}{q}\right)\\
\end{eqnarray*}
The inequality on the right clearly holds for $q\geq4$, since $\sqrt{3},\lambda>1$ and $\cos(\frac{\pi}{q})\geq\sin(\frac{\pi}{q})$.  If $q=3$, the last inequality reduces to $\lambda>1$, which is true. Next, we show that $b\leq c$, and since we have established that $b>0$, we will have $c>0$.  A straight forward computation using (\ref{sumlambda}) and (\ref{E}) yields:
\begin{eqnarray*}
b\leq c 
&\Leftrightarrow& \cos\left(\frac{\pi}{q}\right)\leq\cos\left(\frac{\pi}{r}\right)\\
\end{eqnarray*}

As $3\leq q\leq r$ the inequality on the right holds, and it is an equality if and only if $q=r$.  This finishes the proof since $0<b\leq c$ implies that $\sqrt{\frac{b}{c}}\leq1$.
\end{proof}

\indent  Among the elements $BA^{-1},BC^{-1}$, and $AC^{-1}$ in $\Gamma(p,q,r)$ which are hyperbolic, we want to know what the projection of their axes look like.  To achieve this we choose appropriate fundamental domains for each of the elements (Figures \ref{reflectBA}-\ref{reflectAC}). The choices were made in order to represent the elements as a product of two reflections in geodesics containing sides of the fundamental quadrilateral.  In each figure, and for all figures in the proof of Theorem \ref{classthm}, we label the relevant vertices of the fundamental domain of $\Gamma(p,q,r)$ with the elliptic isometry fixing that point.  Thus, the label $A$ is at the point $i$ and the label $B$ is at the point $\lambda^{-1}i$.  For Figures \ref{reflectBA}-\ref{reflectAC}, let $\sigma_i$ be the reflection in the geodesic $L_i$ containing the side $s_i$, for $i=1,2,3$.  In Figure \ref{reflectBA}, we can write $B=\sigma_1\sigma_2$ and $A^{-1}=\sigma_2\sigma_3$, and so $BA^{-1}=\sigma_1\sigma_2\sigma_2\sigma_3=\sigma_1\sigma_3$. Hence, if $L_1$ and $L_3$ are disjoint in $\overline{\mathbb{H}}=\mathbb{H}\cup\mathbb{R}\cup\{\infty\}$, then $BA^{-1}$ is hyperbolic and $\mathcal{A}_{BA^{-1}}$ is the common orthogonal to $L_1$ and $L_3$.  
The key idea in all three figures was to choose $L_2$ to be the unique geodesic which passes through the fixed points of the two elliptic isometries in the composition.  For example for $BA^{-1}$ we chose $L_2$ to be the imaginary axis which is the unique geodesic passing through the points $i$ and $\lambda^{-1}i$, the fixed points of $A^{-1}$ and $B$, respectively. \\
\indent In order to state and prove the main theorem in a less cumbersome way, we adopt the following notation.  If $f$ is conjugate to $g$ in $\Gamma(p,q,r)$, we denote this by $f\sim g$.  Also, we interpret an order $\infty$ orbifold point as a puncture.

\begin{figure}
\centering
\begin{overpic}[scale=.3]{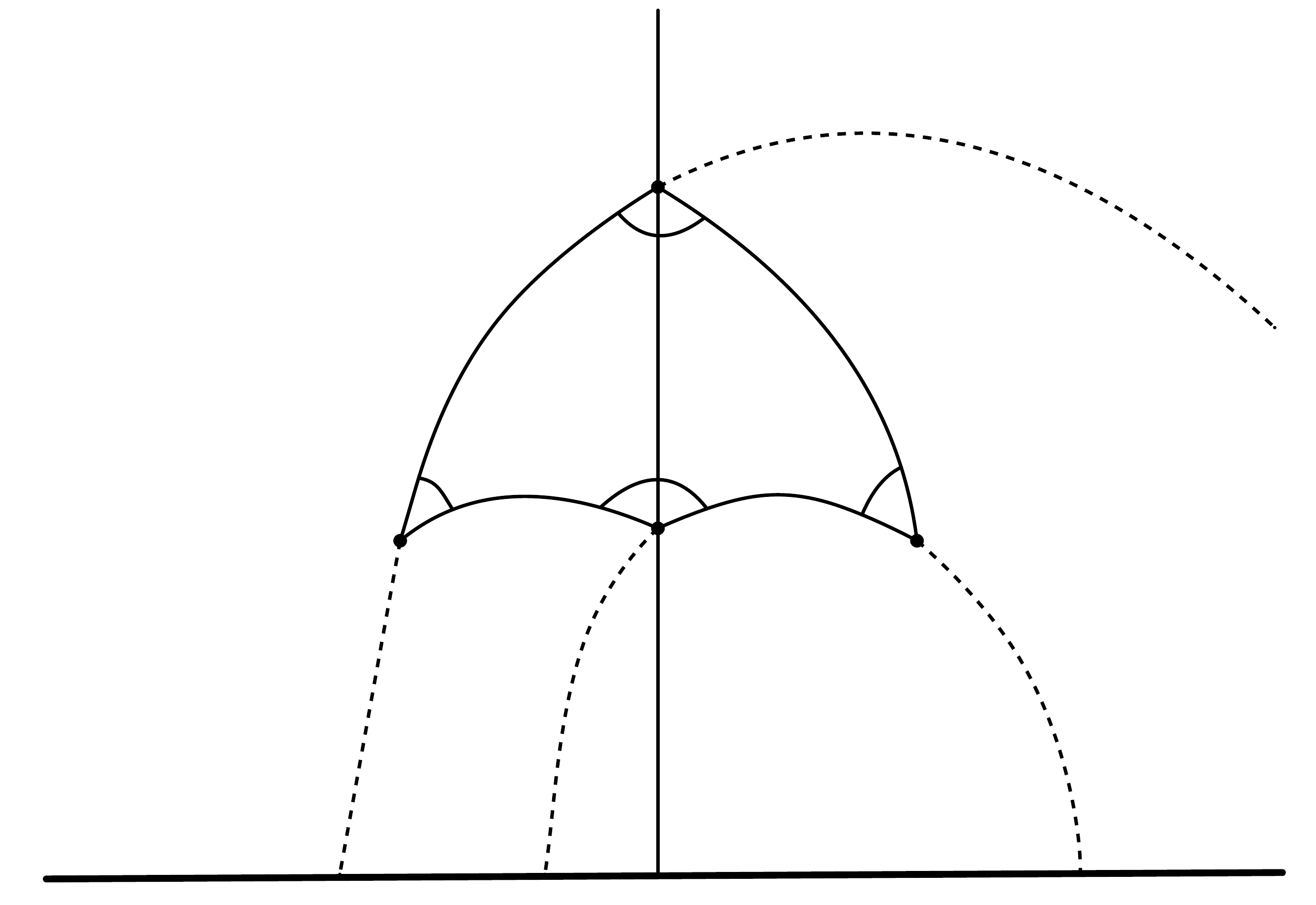}
\put(25.5,25){\tiny{$C$}}
\put(46,55){\tiny{$A$}}
\put(50.5,25){\tiny{$B$}}
\put(33,33){\tiny{$\frac{\pi}{r}$}}
\put(34.5,45){\tiny{$s_3$}}
\put(47,34){\tiny{$\frac{\pi}{q}$}}
\put(50.5,34){\tiny{$\frac{\pi}{q}$}}
\put(47,48){\tiny{$\frac{\pi}{p}$}}
\put(50.5,48){\tiny{$\frac{\pi}{p}$}}
\put(65,60){\tiny{$L_3$}}
\put(50.5,40.5){\tiny{$s_2$}}
\put(50.5,12){\tiny{$L_2$}}
\put(58,29){\tiny{$s_1$}}
\put(65,33){\tiny{$\frac{\pi}{r}$}}
\put(78.5,20){\tiny{$L_1$}}
\put(49.1,-1){\tiny{\footnotesize0}}
\put(90,65){\tiny{$\mathbb{H}$}}
\end{overpic}
\caption{Fundamental domain of $\Gamma(p,q,r)$ chosen for $BA^{-1}$}
\label{reflectBA}
\end{figure}

\begin{figure}
\centering
\begin{overpic}[scale=.3]{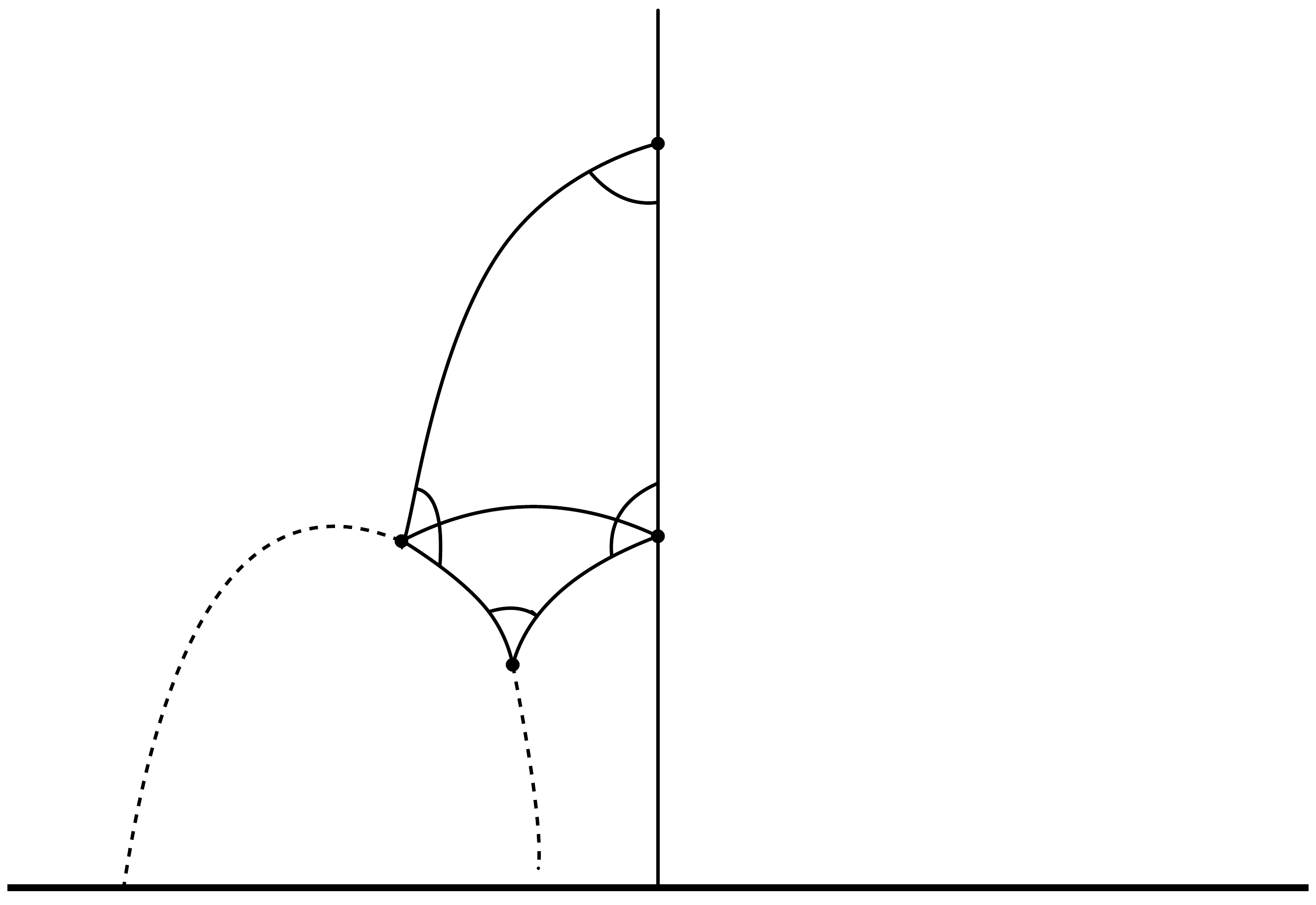}
\put(9,20){\tiny{$L_3$}}
\put(26,25){\tiny{$C$}}
\put(32,22){\tiny{$s_3$}}
\put(33,31.5){\tiny{$\frac{\pi}{r}$}}
\put(33.4,26.5){\tiny{$\frac{\pi}{r}$}}
\put(38,24.5){\tiny{$\frac{\pi}{p}$}}
\put(43,27){\tiny{$\frac{\pi}{q}$}}
\put(45,32){\tiny{$\frac{\pi}{q}$}}
\put(45,50.5){\tiny{$\frac{\pi}{p}$}}
\put(39,31){\tiny{$s_2$}}
\put(51,57){\tiny{$A$}}
\put(51,44){\tiny{$s_1$}}
\put(51,27){\tiny{$B$}}
\put(51,13){\tiny{$L_1$}}
\put(49.2,-2){\footnotesize0}
\put(90,65){\tiny{$\mathbb{H}$}}
\end{overpic}
\caption{Fundamental domain of $\Gamma(p,q,r)$ chosen for $BC^{-1}$}
\label{reflectBC}
\end{figure}

\begin{figure}
\centering
\begin{overpic}[scale=.3]{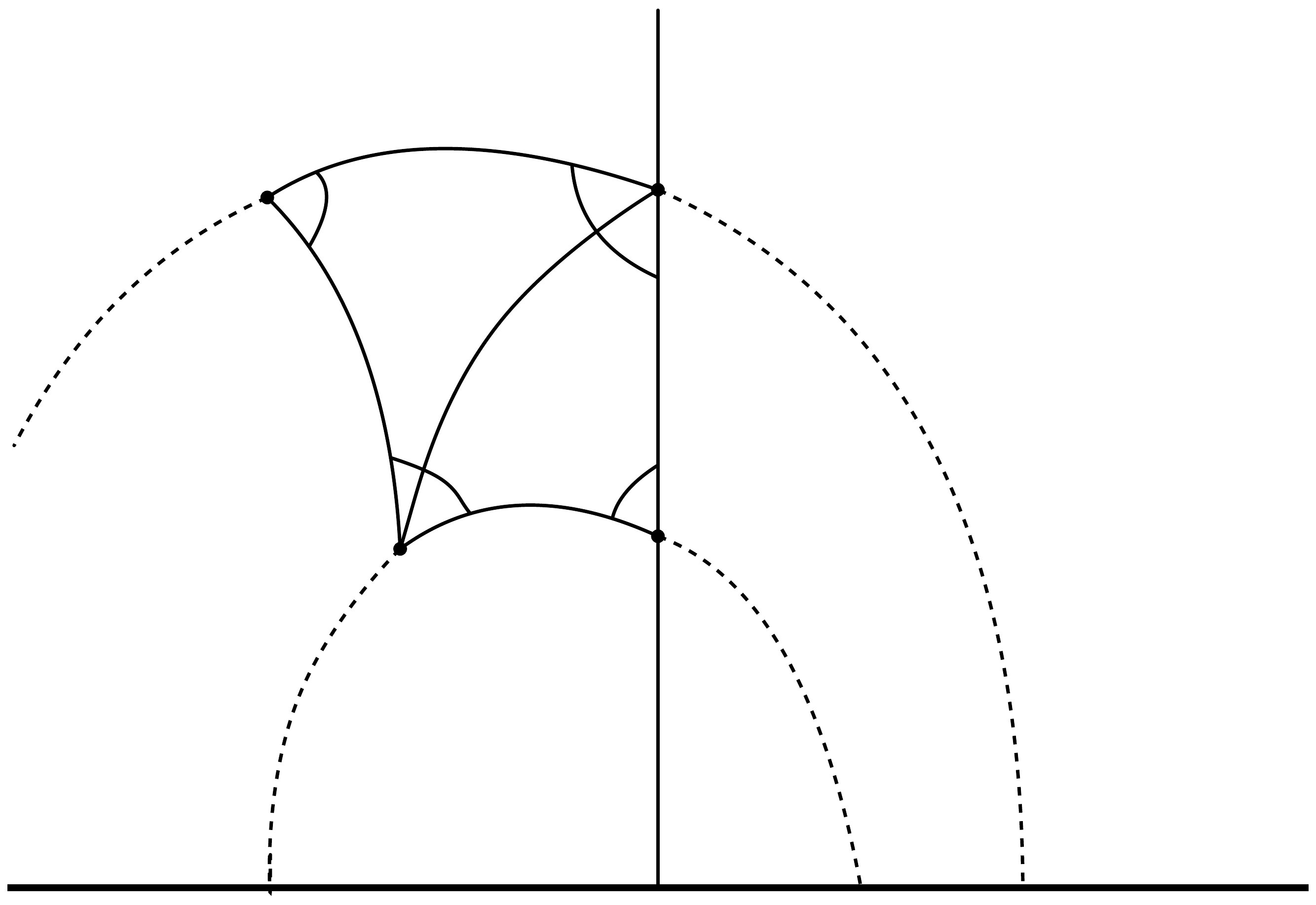}
\put(25.5,52.5){\tiny{$\frac{\pi}{q}$}}
\put(35,59){\tiny{$s_1$}}
\put(41,52.5){\tiny{$\frac{\pi}{p}$}}
\put(51,55){\tiny{$A$}}
\put(31,24){\tiny{$C$}}
\put(30,35.5){\tiny{$\frac{\pi}{r}$}}
\put(38,42){\tiny{$s_2$}}
\put(45.5,46){\tiny{$\frac{\pi}{p}$}}
\put(34.5,33.1){\tiny{$\frac{\pi}{r}$}}
\put(39,27.5){\tiny{$s_3$}}
\put(45.5,33.5){\tiny{$\frac{\pi}{q}$}}
\put(46,25){\tiny{$B$}}
\put(70,40){\tiny{$L_1$}}
\put(61,20){\tiny{$L_3$}}
\put(49.2,-2){\tiny{\footnotesize0}}
\put(90,65){\tiny{$\mathbb{H}$}}
\end{overpic}
\caption{Fundamental domain of $\Gamma(p,q,r)$ chosen for $AC^{-1}$}
\label{reflectAC}
\end{figure}
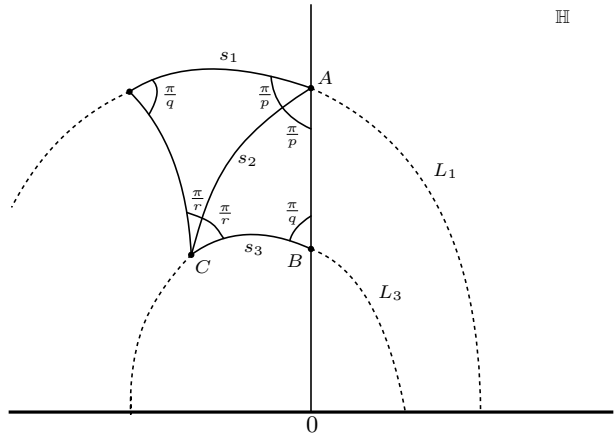

\begin{thm}
Let $\pi:\mathbb{H}\rightarrow\mathcal{O}(p,q,r)=\mathbb{H}/\Gamma(p,q,r)$ be the projection map and assume $2\leq p\leq q\leq r\leq\infty$.
\label{classthm}
\begin{itemize}
\item[($i$)]
In $\Gamma(2,q,r)$ for $q\geq3$:
\begin{itemize}
  \item[$\bullet$] $BA^{-1}$ is elliptic of order $r$ (parabolic if $r=\infty$).
  \item[$\bullet$] If $q=3$, then $BC^{-1}$ is elliptic of order $r$ (parabolic if $r=\infty$); if $q=4$, then $BC^{-1}$ is hyperbolic and $\pi(\mathcal{A}_{BC^{-1}})$ is the geodesic path from the orbifold point of order two to the orbifold point of order four and back again;  if $q\geq5$, then $BC^{-1}$ is hyperbolic and $\pi(\mathcal{A}_{BC^{-1}})$ is the figure eight geodesic bounding the orbifold points of orders $q$ and $r$.
\item[$\bullet$] $AC^{-1}$ is elliptic of order $q$ (parabolic if $q=\infty$).
\end{itemize}
\item[($ii$)]
In $\Gamma(3,q,r)$ for $q\geq3,r\geq4$:
\begin{itemize}
\item[$\bullet$] $BA^{-1}\sim AC^{-1}$ and they are hyperbolic.  If $q=r$, then $\pi(\mathcal{A}_{BA^{-1}})=\pi(\mathcal{A}_{AC^{-1}})$ is the geodesic path passing through the orbifold point of order three pictured in Figure \ref{3qq}; if $q\neq r$, then $\pi(\mathcal{A}_{BA^{-1}})=\pi(\mathcal{A}_{AC^{-1}})$ is the figure eight geodesic bounding the orbifold points of orders three and $q$.
\item[$\bullet$] If $q=3$,  then $BC^{-1}\sim (BA^{-1})^{-1}$ and they are hyperbolic.  Thus, $\pi(\mathcal{A}_{BC^{-1}})=\pi\left(\mathcal{A}_{(BA^{-1})^{-1}}\right)$, which is simply $\pi(\mathcal{A}_{BA^{-1}})$ but with opposite orientation.  However, if $q\geq4$, then $BC^{-1}$ is hyperbolic and $\pi(\mathcal{A}_{BC^{-1}})$ is the figure eight geodesic bounding the orbifold points of orders $q$ and $r$.
\end{itemize}
\item[($iii$)]
In $\Gamma(p,q,r)$ for $p,q,r\geq4$:
\begin{itemize}
\item[$\bullet$] $BA^{-1}$ is hyperbolic and $\pi(\mathcal{A}_{BA^{-1}})$ is the figure eight geodesic bounding the orbifold points of orders $p$ and $q$.
\item [$\bullet$] $BC^{-1}$ is hyperbolic and $\pi(\mathcal{A}_{BC^{-1}})$ is the figure eight geodesic bounding the orbifold points of orders $q$ and $r$.
\item[$\bullet$] $AC^{-1}$ is hyperbolic and $\pi(\mathcal{A}_{AC^{-1}})$ is the figure eight geodesic bounding the orbifold points of orders $p$ and $r$.
\end{itemize}
\end{itemize}
\end{thm}

\begin{proof}
We note that in this proof all of the figures are drawn for the cases where $p,q,$ and $r$ are finite, but the figures and corresponding arguments can be easily adjusted in the infinite cases.

\begin{component}[Case (i): $\Gamma(2,q,r)$ for $q\geq3$] First note that $BA^{-1}=BA\sim AB=C^{-1}$, an elliptic isometry of order $r$ (parabolic if $r=\infty$).  Thus, $BA^{-1}$ is elliptic of order $r$ (parabolic if $r=\infty$).\\
\indent Now, if $q=3$, we see that $BC^{-1}=BAB\sim B^2A=B^{-1}A^{-1}=C$, an elliptic isometry of order $r$ (parabolic if $r=\infty$). Next, if $q=4$, $BC^{-1}=BAB\sim B^2A$, which is the product of two non-conjugate order two elliptic isometries.  It follows that $BC^{-1}$ is hyperbolic and its axis projects to the geodesic path from the order two orbifold point to the order four orbifold point and back again.  If $q\geq5$, we work with the fundamental domain in Figure \ref{BC(two,q,r)}. As was done above in Figure \ref{reflectBC}, here in Figure \ref{BC(two,q,r)} we express $BC^{-1}$ as the composition of reflection in the imaginary axis with reflection in the dashed geodesic, $L$.  Note that since $q,r\geq5$, we get the lower bound $\frac{3\pi}{5}$ for the angles so indicated in the figure.  It is clear then that $L$ and the imaginary axis do no intersect in $\overline{\mathbb{H}}$ because otherwise a triangle is formed whose sum of interior angles is greater than $\pi$. So indeed, $BC^{-1}$ is hyperbolic and $\mathcal{A}_{BC^{-1}}$ is the common orthogonal of $L$ and the imaginary axis.  Furthermore, we can see that $\mathcal{A}_{BC^{-1}}$ must pass through the sides of the fundamental quadrilateral, disjoint from the vertices (the red geodesic segment), since all other possibilities result in a polygon whose interior angle sum is too large to be hyperbolic.  The red segment of $\mathcal{A}_{BC^{-1}}$ represents only half of the translation length, but a symmetric copy lies on the other side of the imaginary axis.  When translating that other segment into the original fundamental quadrilateral by $B^{-1}$ and then passing to the quotient, we get the figure eight geodesic whose component loops bound the orbifold points of orders $q$ and $r$.\\ 
\indent Finally, $AC^{-1}=A^2B=B$, an order $q$ elliptic isometry (parabolic if $q=\infty$).

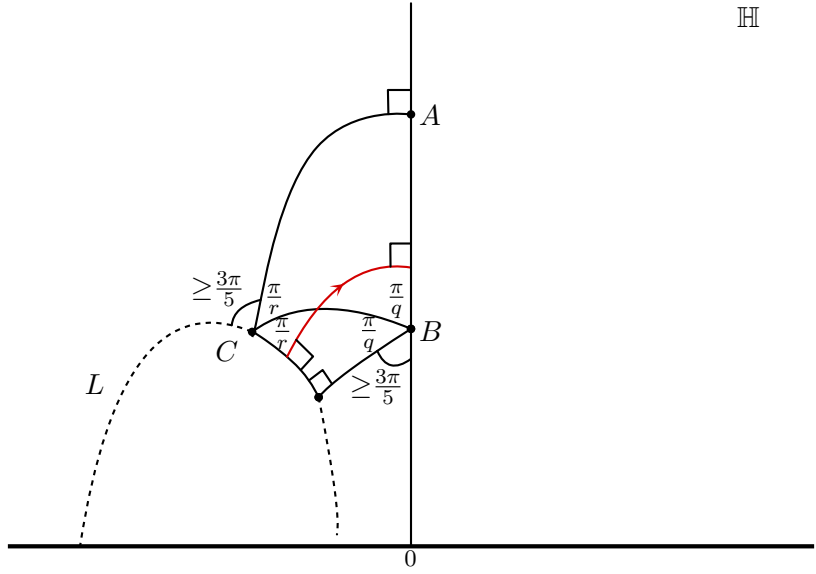
\begin{figure}
\centering
\begin{overpic}[scale=.4]{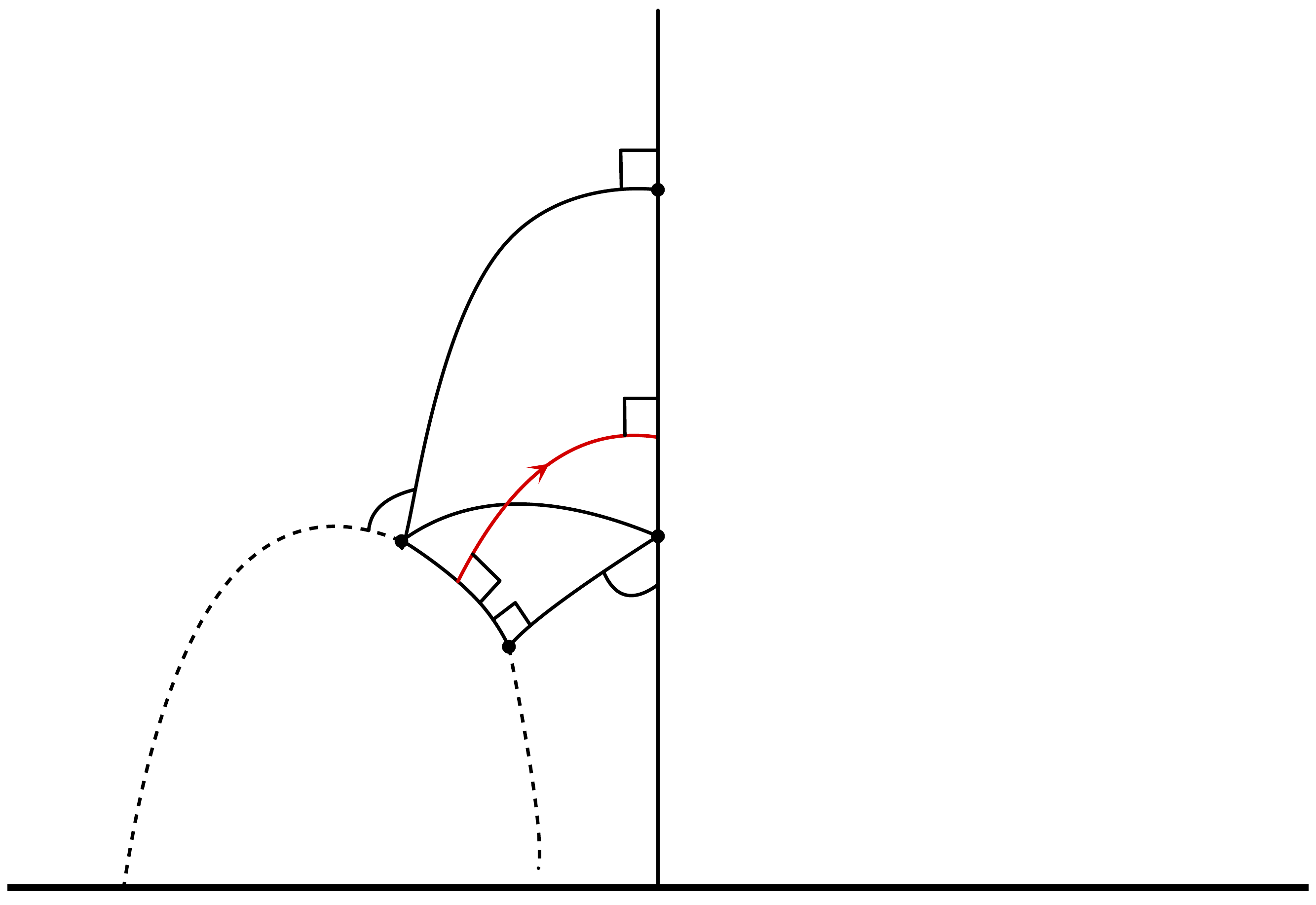}
\put(23,32){${\geq}\frac{3\pi}{5}$}
\put(26,24){$C$}
\put(32,31){$\frac{\pi}{r}$}
\put(33,26.8){$\frac{\pi}{r}$}
\put(42.5,20){${\geq}\frac{3\pi}{5}$}
\put(51,53){$A$}
\put(51,26.5){$B$}
\put(47,31){$\frac{\pi}{q}$}
\put(43.6,26.9){$\frac{\pi}{q}$}
\put(49.2,-1.1){\footnotesize0}
\put(90,65){$\mathbb{H}$}
\put(10,20){$L$}
\end{overpic}
\caption{A segment of $\mathcal{A}_{BC^{-1}}$ where $BC^{-1}\in\Gamma(2,q,r)$ for $q,r\geq5$}
\label{BC(two,q,r)}
\end{figure}
\end{component}

\begin{component}[Case (ii): $\Gamma(3,q,r)$ for $q\geq3, r\geq4$] Well, $BA^{-1}\sim A^{-1}B=A^2B=AC^{-1}$, and we choose to consider $BA^{-1}$.  The corresponding fundamental quadrilateral for $\Gamma(3,q,r)$ for $q\geq3, r\geq4$ is in Figure \ref{3qqdom}.  The two dashed geodesics in the figure cannot intersect and so $BA^{-1}$ is hyperbolic.  After ruling out potential geodesic segments of $\mathcal{A}_{BA^{-1}}$ which create polygons whose interior angle sums are too large to be hyperbolic, we are left with the three geodesic segments labeled $\alpha, \beta,$ and $\gamma$.  Here $\alpha$ is representing a geodesic segment with end points on the sides $s_1$ and $s_2$ of the quadrilateral (disjoint from the vertices), while $\beta$ has an end point at $i$ and an end point on the side $s_2$ (disjoint from the vertices), and finally $\gamma$ has an end point on the dashed geodesic segment $u$ and the side $s_2$ (disjoint from the vertices). By Lemma \ref{BAlem}, we know that if $q=r$, $\mathcal{A}_{BA^{-1}}$ intersects the imaginary axis at $i$ and thus $\beta$ is the correct segment.  It represents half of the translation length of $BA^{-1}$.  The other half is just the reflection of $\beta$ in the imaginary axis with opposite orientation.  In the quotient, we will have a geodesic loop passing through the order three orbifold point (see Figure \ref{3qq}).  If $q\neq r$, then the correct geodesic segment in Figure \ref{3qqdom} is $\alpha$, since by Lemma \ref{BAlem},  $\mathcal{A}_{BA^{-1}}$ must pass through the imaginary axis strictly below $i$. (The segment $\gamma$ cannot intersect the imaginary axis below $i$ because it would yield a triangle whose interior angles sum to at least $\pi$.)  It follows that $\mathcal{A}_{BA^{-1}}$ projects to the figure eight geodesic bounding the orbifold points of orders three and $q$, when $q\neq r$.

\begin{figure}
\centering
\begin{overpic}[scale=.4]{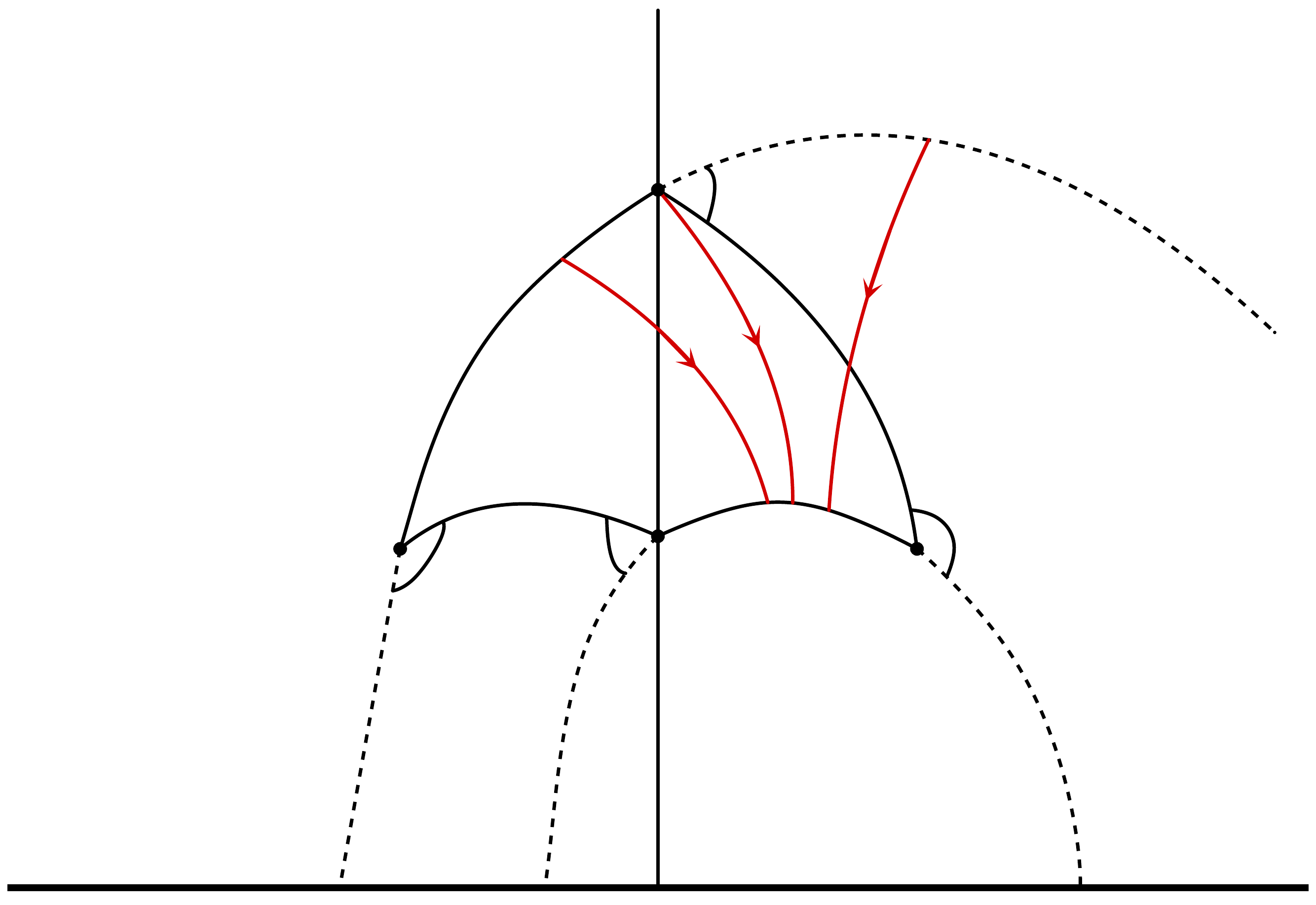}
\put(26,26){$C$}
\put(32.5,23.5){${\geq}\frac{3\pi}{4}$}
\put(41,26){${\geq}\frac{\pi}{3}$}
\put(51,25){$B$}
\put(34,44){$s_1$}
\put(59,27){$s_2$}
\put(46,55){$A$}
\put(44.5,43){$\alpha$}
\put(52,46){$\beta$}
\put(55.2,53){$\frac{\pi}{3}$}
\put(64.5,49){$\gamma$}
\put(80,57){$u$}
\put(73,28){${\geq}\frac{3\pi}{4}$}
\put(49.2,-1.1){\footnotesize0}
\put(90,65){$\mathbb{H}$}
\end{overpic}
\caption{Possible segments of $\mathcal{A}_{BA^{-1}}$ for $BA^{-1}\in\Gamma(3,q,r)$, $q\geq3,r\geq4$}
\label{3qqdom}
\end{figure}

\begin{figure}
\begin{overpic}[scale=.4]{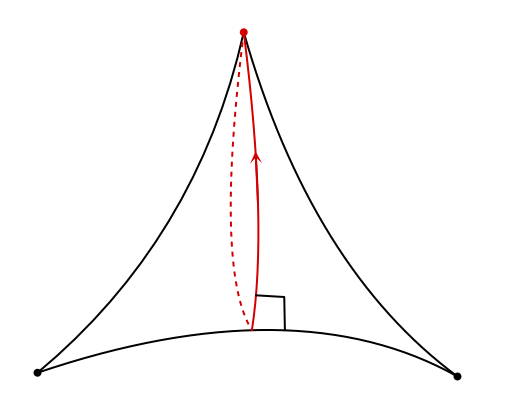}
\put(45,73){$3$}
\put(91.7,1){$q$}
\put(1,1.5){$q$}
\end{overpic}
\caption{$\pi(\mathcal{A}_{BA^{-1}})\subset\mathcal{O}(3,q,q)$ for $q\geq 4$}
\label{3qq}
\end{figure}

\indent If $q=3$, $BC^{-1}=BAB\sim B^2A=B^{-1}A\sim AB^{-1}=(BA^{-1})^{-1}$, which is hyperbolic by the above discussion.  Thus, $BC^{-1}$ is hyperbolic and $\pi(\mathcal{A}_{BC^{-1}})$ is also a figure eight geodesic bounding the orbifold points of order three, but with opposite orientation as $\pi(\mathcal{A}_{BA^{-1}})$.  If $q\geq4$, then necessarily $r\geq4$, and a similar analysis to cases above (e.g. $BC^{-1}\in\Gamma(2,q,r)$ for $q,r\geq5$) shows that $BC^{-1}$ is hyperbolic and $\pi(\mathcal{A}_{BC^{-1}})$ is the figure eight geodesic bounding the orbifold points of orders $q$ and $r$. 

\end{component}

\begin{component}[Case (iii): $\Gamma(p,q,r)$ for $p,q,r\geq4$] By choosing the appropriate fundamental domains for each of the three isometries, one sees that each isometry is hyperbolic and the axis of each isometry must pass through opposite sides of the quadrilateral disjoint from the vertices.  As in above cases, this is done by showing any other proposed common orthogonal will form a polygon whose interior angle sum is too large.
\end{component}\end{proof}

Theorem \ref{classthm} yields a classification of figure eight geodesics on triangle group orbifolds.  In addition, we can deduce two corollaries.

A triangle group is called \textit{exceptional} if it has the form
$\Gamma(2,3,r)$ or $\Gamma(2,4,r)$.  Note that $\Gamma(2,3,r)$ is a triangle group if and only if $r\geq7$ and $\Gamma(2,4,r)$ is a triangle group if and only if $r\geq5$.

\begin{cor}
\label{exceptg}
If $\Gamma(p,q,r)$ is not an exceptional triangle group, then $\mathcal{O}(p,q,r)$ contains a figure eight geodesic.
\end{cor}

\begin{cor}
Let $\mathcal{C}_{pq}, \mathcal{C}_{qr}$, and $\mathcal{C}_{pr}$ be the free $\mathcal{O}(p,q,r)$-homotopy classes described at the beginning of this section.  If these free $\mathcal{O}(p,q,r)$-homotopy classes contain a closed geodesic, then that geodesic will have one of the following three forms:
\begin{itemize}
\item[(i)]
a figure eight geodesic
\item[(ii)]
the geodesic path from an orbifold point of order two to an orbifold point of order four and back again.
\item[(iii)]
the geodesic passing through the order three orbifold point pictured in Figure \ref{3qq}.
\end{itemize}
\end{cor}

\subsection{The Shortest Figure Eight Geodesics}
Now that we have identified the figure eight geodesics on triangle group orbifolds we set out to find the shortest one.  Suppose $\Gamma$ is a Fuchsian group and let $\pi:\mathbb{H}\rightarrow\mathbb{H}/\Gamma$ be the projection map.  If $g\in\Gamma$ is hyperbolic then the \textit{length of the closed geodesic $\pi(\mathcal{A}_g)$} on $\mathbb{H}/\Gamma$ is the translation length $\mathcal{T}_g$ of $g$.  Recall the following well-known formula for a hyperbolic element $g\in$ PSL$(2,\mathbb{R})$ (see \cite{AB1}, for example).
\begin{equation}
\label{trtrl}
|tr(g)|=2\cosh(\mathcal{T}_g/2)
\end{equation}
Thus, $|tr(g)|$ and $\mathcal{T}_g$ have a direct relationship.  We will use the absolute value of the trace in proofs, but we state our results in terms of hyperbolic length.  From now on we will write $|$trace$|$ in place of the phrase ``absolute value of the trace."  In all of the proofs below we use the normalized triangle group $\Gamma(p,q,r)$ as in Section \ref{tgsection} and we assume $2\leq p\leq q \leq r\leq\infty$.  In addition, for the remainder of this article unless otherwise stated, closed geodesics are taken to be unoriented.

\begin{lem}
\label{lem1}
The shortest figure eight geodesic on a triangle group orbifold without punctures is the unique one on $\mathcal{O}(3,3,4)$ which bounds the orbifold points of order three and has length $$2\cosh^{-1}\left(\frac{1+\sqrt{2}}{2}\right)=1.26595\dots.$$
\end{lem}

\begin{proof}
We need only consider those elements $BA^{-1}, BC^{-1}$, and $AC^{-1}$ in $\Gamma(p,q,r)$ which have an axis that projects to a figure eight geodesic (see Theorem \ref{classthm}).  Recalling the trace identity derived from the Cayley-Hamilton Theorem above, (\ref{treqn}), we have:
\begin{eqnarray}
\label{BAtrace}
|tr(BA^{-1})|&=&2\left[2\cos\left(\frac{\pi}{p}\right)\cos\left(\frac{\pi}{q}\right)+\cos\left(\frac{\pi}{r}\right)\right],\\
\label{BCtrace}
|tr(BC^{-1})|&=&2\left[2\cos\left(\frac{\pi}{q}\right)\cos\left(\frac{\pi}{r}\right)+\cos\left(\frac{\pi}{p}\right)\right],\\
\label{ACtrace}
|tr(AC^{-1})|&=&2\left[2\cos\left(\frac{\pi}{p}\right)\cos\left(\frac{\pi}{r}\right)+\cos\left(\frac{\pi}{q}\right)\right].
\end{eqnarray}

To find the smallest $|$trace$|$ we consider two cases.  First, for each triangle group of the form $\Gamma(2,q,r)$, the only element which has an axis projecting to a figure eight geodesic is $BC^{-1}$, and only if $q,r\geq5$. From (\ref{BCtrace}) we see that $|tr(BC^{-1})|=4\cos(\frac{\pi}{q})\cos(\frac{\pi}{r})$ where $BC^{-1}\in\Gamma(2,q,r)$ for $q,r\geq5$, and thus $|tr(BC^{-1})|$ will increase as $q$ and/or $r$ increase.  It follows that $BC^{-1}\in\Gamma(2,5,5)$ has the smallest $|$trace$|$, namely $4\cos^2(\frac{\pi}{5})=2.618\dots$.\\
\indent Next, we look at triangle groups of the form $\Gamma(p,q,r)$ for $p,q\geq 3$ and $r\geq4$.  This includes all remaining triangle groups with elements having axes projecting to figure eight geodesics.  We note that all three elements ($BA^{-1}, BC^{-1}$, and $AC^{-1}$) in $\Gamma(3,3,4)$ have the same $|$trace$|$, $1+\sqrt{2}=2.414\dots$.  Observing (\ref{BAtrace})-(\ref{ACtrace}), we see that as any of $p,q,$ and/or $r$ increase, the $|$trace$|$ will increase.  Thus $1+\sqrt{2}$ is the smallest $|$trace$|$ in this case.\\
\indent Comparing the results of the two cases we find that the unique figure eight geodesic on $\mathcal{O}(3,3,4)$ is the shortest.  Using that the $|$trace$|$ of the corresponding hyperbolic elements is $1+\sqrt{2}$, we get the resulting hyperbolic length of the figure eight geodesic using (\ref{trtrl}).
\end{proof}

Similar arguments to those used in the proof of Lemma \ref{lem1} can be used to establish the next three lemmas.  We note that Lemma \ref{lem4} is already well-known \cite{PB}, \cite{AY}.

\begin{lem}
\label{lem2}
The shortest figure eight geodesic on a triangle group orbifold with one puncture is the unique one on $\mathcal{O}(3,3,\infty)$ which bounds the orbifold points of order three and has length $$2\cosh^{-1}\left(\frac{3}{2}\right)=1.92485\dots.$$
\end{lem}

\begin{lem}
\label{lem3}
The shortest figure eight geodesic on a triangle group orbifold with two punctures is the unique one on $\mathcal{O}(2,\infty,\infty)$ which bounds the punctures and has length $$2\cosh^{-1}(2)=2.63392\dots.$$
\end{lem}

\begin{lem}
\label{lem4}
$\mathcal{O}(\infty,\infty,\infty)$ contains three figure eight geodesics all having length $$2\cosh^{-1}(3)=4\ln(1+\sqrt{2})=3.52549\dots.$$
\end{lem}

\begin{thm}
The shortest figure eight geodesic on a triangle group orbifold containing at least one puncture is the unique one on $\mathcal{O}(3,3,\infty)$ which bounds the orbifold points of order three and has length $$2\cosh^{-1}\left(\frac{3}{2}\right)=1.92485\dots.$$
\end{thm}

\begin{proof}[Proof of Theorem 1.1]
Lemmas \ref{lem1}-\ref{lem4} establish the claim.
\end{proof}

\begin{rem}
Among those triangle groups containing figure eight geodesics, the one containing the shortest figure eight geodesic, $\mathcal{O}(3,3,4)$, has the smallest area: $\frac{\pi}{6}$.
\end{rem}

\section{Figure Eight Geodesics on 2-Orbifolds}
With Theorem \ref{shortfig8tg} established we look to generalize by finding the shortest figure eight geodesic on any 2-orbifold without orbifold points of order two.  We begin with a proposition which highlights one of the reasons order two orbifold points are exceptional.

\begin{prop}(\cite{BM},Chapter V, Proposition G.9)
\label{Maskitlem}
 Let $\Gamma$ be a Fuchsian group and let~$\pi:\mathbb{H}\rightarrow\mathbb{H}/\Gamma$ be the projection map.  Suppose $\alpha$ is an essential simple loop on $\mathbb{H}/\Gamma$ which is disjoint from the orbifold points and is freely $\mathbb{H}/\Gamma$-homotopic to a closed geodesic.  If $\alpha$ does not bound a disc containing exactly two orbifold points both of order two, then the closed geodesic corresponding to $\alpha$ is also simple. If $\alpha$ does bound a disc containing exactly two orbifold points of order two, then the corresponding closed geodesic is the path which runs from one orbifold point of order two to the other orbifold point of order two, and back again.
\end{prop}

\indent It is a well-established fact that there is a non-elementary finitely generated Fuchsian group $\Gamma$ with the signature $(g:m_1,\dots,m_c;n;b)$ if and only if
\begin{equation}
\label{siginequal}
2g-2+n+b+\displaystyle\sum_{j=1}^c\left(1-\frac{1}{m_j}\right)>0.
\end{equation}

\begin{defi}
Let $\Gamma$ be a Fuchsian group.  The 2-orbifold $\mathbb{H}/\Gamma$ is called a \textit{generalized pair of pants} if $\Gamma$ has signature $(0:m_1,\dots,m_c;n;b)$ satisfying (\ref{siginequal}) 
and $c+n+b=3$.
\end{defi}
Thus triangle group orbifolds, as well as a usual pair of pants (a sphere with three boundary components), are examples of generalized pairs of pants.\\

\begin{prop}
\label{fig8pants}
A figure eight geodesic on a 2-orbifold without orbifold points of order two is contained in a generalized pair of pants.
\end{prop}

\begin{proof}
Suppose $\Gamma$ is a Fuchsian group and let $\gamma$ be a figure eight geodesic on $\mathbb{H}/\Gamma$.  Choose the point of self-intersection of $\gamma$, say $x$, to be the base point of $\mathbb{H}/\Gamma$.  Then we can write $\gamma=\alpha*\beta$, where $\alpha$ and $\beta$ are the two component simple loops of $\gamma$ based at $x$.  Each of $\alpha$ and $\beta$ either bound an orbifold point, bound a puncture, or are freely $\mathbb{H}/\Gamma$-homotopic to a simple closed geodesic (Theorem \ref{uniquegeodthm}).  Note that neither $\alpha$ nor $\beta$ can be freely $\mathbb{H}/\Gamma$-homotopic to the geodesic path between two orbifold points of order two (see Proposition \ref{Maskitlem}) because $\mathbb{H}/\Gamma$ does not contain any orbifold points of order two.  Let $\Phi:\Pi_1(\mathbb{H}/\Gamma,x)\rightarrow\Gamma$ be the corresponding isomorphism from Theorem \ref{isomorphthm} and assume $\Phi([\alpha])=f$ and $\Phi([\beta])=g$.  Thus, $\Phi([\gamma])=fg$.  Then, the conjugacy class of $fg^{-1}$ represents a free $\mathbb{H}/\Gamma$-homotopy class containing a simple loop, which either bounds an orbifold point, bounds a puncture, or is freely $\mathbb{H}/\Gamma$-homotopic to a simple closed geodesic (based on whether $fg^{-1}$ is respectively, elliptic, parabolic, or hyperbolic).  Thus $\gamma$ is contained in a generalized pair of pants with the three components being the projections of the fixed point (if elliptic or parabolic) or the axis (if hyperbolic) of $f$, $g$, and $fg^{-1}$.
\end{proof}\\

\begin{proof}[Proof of Theorem 1.2] Suppose $\Gamma$ is a Fuchsian group and let $\gamma$ be a figure eight geodesic on $\mathbb{H}/\Gamma$.  By Proposition \ref{fig8pants}, $\gamma$ is contained in a generalized pair of pants $P$ within $\mathbb{H}/\Gamma$.  If $P$ has no boundary components, then $P=\mathbb{H}/\Gamma$ is a triangle group orbifold and so the result follows from Theorem \ref{shortfig8tg}. So suppose $P$ contains at least one boundary component.  As we decrease the length of the boundary components of $P$ so they become punctures, the lengths of loops on $P$ will decrease (see \cite{WT}, Lemma 3.4).  Thus, the smallest figure eight geodesic cannot be contained in a generalized pair of pants with a boundary component.  It follows that it must lie on a triangle group orbifold and so again the result follows from Theorem \ref{shortfig8tg}.
\end{proof}

Table \ref{tab1} compares known results with the main result of this note, Theorem \ref{shortfig8ho}.\\

\begin{table}
\centering
\caption{Minimal length non-simple closed geodesics}
\hspace*{-2cm}
\begin{tabular}{| >{\centering\arraybackslash}m{3.1cm} |>{\centering\arraybackslash}m{3cm} |c|>{\centering\arraybackslash}m{2cm} |c|>{\centering\arraybackslash}m{1.2cm} |}
 \hline
 \small{Classification} & \small{Closed Geodesic(s)} & $|$trace$|$ & \small{Hyperbolic length} & \small{Orbifold} & \small{Area of Orbifold} \\[.1in]\hline
   \small{Shortest non-simple closed geodesic on a 2-orbifold (\cite{TN,CPNP,RV})} & \small{Passes through order two orbifold point (Figure \ref{(2,3,7)})} & $2\cos(\frac{2\pi}{7})+1=2.24698\dots$ & 0.98399$\dots$ & $\mathcal{O}(2,3,7)$ & $\frac{\pi}{21}$ \\ \hline
   \small{Shortest figure eight geodesic on a 2-orbifold without order two orbifold points (Theorem \ref{shortfig8ho})} & \small{Figure eight geodesic bounding orbifold points of order three} & $1+\sqrt{2}=2.41421\dots$ & $1.26595\dots$ & $\mathcal{O}(3,3,4)$ & $\frac{\pi}{6}$ \\\hline
   \small{Shortest non-simple closed geodesic on a 2-manifold (\cite{AY0,AY})} & \small{All three figure eight geodesics} & 6 & 3.52549$\dots$ & $\mathcal{O}(\infty,\infty,\infty)$ & $2\pi$ \\ \hline

  \hline
\end{tabular}\hspace*{-2cm}
\label{tab1}
\end{table}

\begin{ack}
The author would like to thank his advisor, Ara Basmajian, for his encouragement and guidance throughout the writing of this work.  In addition, the author thanks Moira Chas for helpful conversations with regards to this work. The results of this article are contained in the Ph.D. Thesis \cite{RSV}.
\end{ack}

\end{document}